\documentclass{amsart}
\title[Algorithms yield upper bounds in differential algebra]{Algorithms yield upper bounds\\ in differential algebra}
\usepackage[numbers]{natbib}
\usepackage{amssymb,amsthm,bm}
\usepackage{enumitem}
\usepackage{caption,color} 

\definecolor{e-mail}{rgb}{0,.40,.80}
\definecolor{reference}{rgb}{.20,.60,.22}
\definecolor{citation}{rgb}{0,.40,.80}

\usepackage[colorlinks, linkcolor=reference,
            citecolor=citation,
          urlcolor=e-mail]{hyperref}

\newtheorem{lemma}{Lemma}[section]
\newtheorem{proposition}[lemma]{Proposition}
\newtheorem{theorem}[lemma]{Theorem}

\newtheorem{corollary}[lemma]{Corollary}
\theoremstyle{definition}
\newtheorem{definition}[lemma]{Definition}
\newtheorem*{claim}{Claim}
\newtheorem{notation}[lemma]{Notation}
\newtheorem{remark}[lemma]{Remark}

\DeclareMathOperator{\trdeg}{trdeg}

\DeclareMathOperator{\Ker}{Ker}

\DeclareMathOperator{\RK}{RK}

\DeclareMathOperator{\rank}{rank}
\DeclareMathOperator{\ord}{ord}

\DeclareMathOperator{\id}{id}
\DeclareMathOperator{\Iter}{Iter}
\DeclareMathOperator{\Len}{Len}
\DeclareMathOperator{\Components}{Components}
\DeclareMathOperator{\KolchinProj}{KolchinProj}

\DeclareMathOperator{\pol}{Pol}
\DeclareMathOperator{\poldim}{PolDim}
\DeclareMathOperator{\size}{Steps}

\DeclareMathOperator{\DCF}{DCF}
\DeclareMathOperator{\RCF}{RCF}
\DeclareMathOperator{\components}{Components}
\DeclareMathOperator{\Char}{char}
  \DeclareMathOperator{\kol}{Kol}
\date{}

\DeclareMathOperator{\lead}{lead}

\DeclareMathOperator{\init}{in}

\author{Wei Li}
\address{KLMM,
   Academy of Mathematics and Systems Science,
   Chinese Academy of Sciences,
    No.55 Zhongguancun East Road, 100190, Beijing, China}
    \email{liwei@mmrc.iss.ac.cn}

\author{Alexey Ovchinnikov}
\address{CUNY Queens College, Department of Mathematics,
65-30 Kissena Blvd, Queens, NY 11367, USA and
CUNY Graduate Center, Mathematics and Computer Science, 365 Fifth Avenue,
New York, NY 10016, USA}
\email{aovchinnikov@qc.cuny.edu}
\author{Gleb Pogudin}
\address{LIX, CNRS, \'Ecole Polytechnique, 
Institut Polytechnique de Paris, 1 rue Honor\'e d'Estienne d'Orves, 91120, Palaiseau, France and
National Research University Higher School of Economics, Moscow, Russia}
\email{gleb.pogudin@polytechnique.edu}
\author{Thomas Scanlon}
\address{University of California, Berkeley, Department of Mathematics, Berkeley, CA 94720-3840, USA}
\email{scanlon@math.berkeley.edu}
\begin{document}
\begin{abstract} 
Consider an algorithm computing in a differential field with several commuting derivations such that the only operations it performs with the elements of the field are arithmetic operations, differentiation, and zero testing.
We show that, if the algorithm is guaranteed to terminate on every input, then there is a computable upper bound for the size of the output of the algorithm in terms of the  size of the input.
We also generalize this to algorithms working with models of good enough theories (including for example, difference fields).

We then apply this to differential algebraic geometry to show that there exists a computable uniform upper bound for the number of components of any variety defined by a system of polynomial PDEs.
We then use this bound to show the existence of a computable uniform upper bound for the elimination problem in systems of polynomial PDEs with delays. 
\end{abstract}
\maketitle

\section{Introduction}
Finding uniform bounds for  problems and quantities (e.g., consistency testing or counting of solutions) is one of the central questions in  differential algebra.
In~\cite{vdDS84}, it was demonstrated that, in commutative algebra, one can show the existence of such bounds as a consequence of theorems about nonstandard extensions of standard algebraic objects.
This approach was successfully applied in the differential algebra context in~\cite{HTKM12} and~\cite[Section~6]{GKOS09} for establishing, for example, the existence of a uniform bound in the differential Nullstellensatz.
Furthermore, in~\cite{ST19}, the authors used methods of proof theory to extract explicit bounds based on nonstandard existence proofs.

The present paper can be viewed as an alternative approach, in which we derive the existence of a computable uniform bound for an object from the existence of an algorithm for computing the object.
More precisely, let $T$ be a 
complete  decidable theory.
The most relevant examples for us would be the theory of differentially closed 
fields in zero characteristic with $m$ commuting derivations and the theory of existentially closed difference fields, others include algebraically closed and real closed fields.
Consider an algorithm $A$ performing computations in a model of $T$ that is restricted to using only definable functions when working with elements of the model (for formal definition, we refer to Section~\ref{subsec:gen_algo}) and required to terminate for every input.

We show that there is a computable upper bound for the size of the output of $A$ in terms of the input size of $A$.
We apply this to the Rosenfeld-Gr\"obner algorithm~\cite{BLOP09}
that decomposes a solution set of a system of polynomial PDEs into components and
is such an algorithm.
This allows us to show that there is a uniform upper bound for the number of components of any differential-algebraic variety defined by a system of polynomial PDEs.
We also show how this bound for the number of components leads to a uniform upper bound for the elimination problem in systems of polynomial PDEs with delays. 

A bound for the number of components of varieties defined by polynomial ODEs appeared in \cite{ordinary_case}, as did a bound for the elimination problem for polynomial ODEs with delays. 
These bounds are based on the application of the Rosenfeld-Gr\"obner algorithm, which, if applied in this situation to ODEs, outputs equations whose order does not exceed the order of the input.
This allowed to restrict to a finitely generated subring of the ring of differential polynomials and use tools from algebraic geometry.
It is non-trivial to generalize this to polynomial PDEs because the orders in the output of the Rosenfeld-Gr\"obner can be  greater than the orders of the input. 
Another key ingredient in the ODE case to obtain the bound in~\cite{ordinary_case} was an analysis of differential dimension polynomials.
A significant difference of our present PDE context with the ordinary case that these polynomials behave less predictably under projections of varieties (compare~\cite[Lemma 6.16]{ordinary_case} and Lemma~\ref{lem:Kolchin_poly_projection}). 
To overcome this difficulty, we use again our bound for the Rosenfeld-Gr\"obner algorithm.

We believe that our 
method
can also be applied to obtain bounds for other algorithms in differential algebra such as~\cite[Algorithm~3.6]{thomas} and for algorithms from other theories, e.g. \cite[Algorithm~3]{GR2019} for systems of difference equations.
Since the reducibility of a polynomial can be expressed as a first-order existential formula, it seems plausible that the same methods could be applied to other algorithms dealing with difference~\cite{GLY09} and differential-difference~\cite{GHYZ2009} equations that use factorization because the corresponding theories satisfy the requirements of our approach~\cite{LS2016,MLS2016,MS2014}.
However, we leave these for future research.

The paper is organized as follows. 
Section~\ref{sec:notation} contains definitions and notation used in Section~\ref{sec:main_results} to state the main results.
Bounds for an algorithm working with a model of a theory $T$ are established in Section~\ref{sec:bound_theories}.
These results are applied to differential algebra in Section~\ref{sec:diffalg}.
Further applications to delay PDEs are given in Section~\ref{sec:delayPDE}.

\section{Basic notions and notaiton}\label{sec:notation}

\begin{definition}[Differential-difference rings]
\begin{itemize}
\item[]
    \item 
 A $\Delta$-$\sigma$-ring $(\mathcal R, \Delta, \sigma)$ is a commutative ring $\mathcal R$ endowed with a finite set  $\Delta=\{\partial_1,\ldots,\partial_m\}$ of commuting derivations of $\mathcal{R}$ and 
 an endomorphism $\sigma$ of $\mathcal{R}$ such that, for all $i$, $\partial_i\sigma=\sigma\partial_i$. 
 \item
 When $\mathcal R$ is additionally a field, it is called a $\Delta$-$\sigma$-field. 
 \item
 If $\sigma$ is an automorphism of $\mathcal R$, $\mathcal R$ is  called a $\Delta$-$\sigma^\ast$-ring.
 \item 
If $\sigma = \operatorname{id}$, $\mathcal{R}$ is called a $\Delta$-ring or differential ring.
 \item For a commutative ring $\mathcal{R}$, $\langle F\rangle$ denotes the ideal generated by $F \subset \mathcal{R}$ in $\mathcal{R}$.
\item For $\Delta=\{\partial_1,\ldots,\partial_m\}$, let $\Theta_\Delta=\{\partial_1^{i_1}\cdot\ldots\cdot\partial_m^{i_m}\mid i_j\geqslant 0,\, 1\leqslant j\leqslant m\}$.
\item For $\theta = \partial_1^{i_1}\cdot\ldots\cdot\partial_m^{i_m}\in\Theta_\Delta$, we let $\ord \theta = i_1+\ldots+i_m$.   For a non-negative integer $B$, we denote $\Theta_\Delta(B):=\{\theta\in\Theta_{\Delta}\mid\, \ord\theta\leqslant B\}$.
 \item For a $\Delta$-ring $\mathcal{R}$, 
 the differential ideal generated by $F\subset \mathcal{R}$ in $\mathcal{R}$ is denoted by 
 $\langle F\rangle^{(\infty)}$; for a non-negative integer $B$,  we introduce the following ideal of $\mathcal{R}$:   \[\langle F\rangle^{(B)} := \langle \theta(F)\mid \,\theta\in\Theta_\Delta(B)\rangle.\]
 \end{itemize}
 \end{definition}

   \begin{definition}[Differential polynomials]
 Let $\mathcal{R}$ be a $\Delta$-ring.  
 The differential polynomial ring over $\mathcal{R}$ in $\bm{y}=y_1,\ldots,y_n$ is defined as
 \[
  \mathcal{R}\{\bm{y}\}_\Delta :=\mathcal{R} [\theta y_s\mid \theta\in\Theta_\Delta;\, 1\leqslant s\leqslant n].
 \]
 The structure of a $\Delta$-ring is defined by $\partial_i(\theta y_s):=(\partial_i\theta)y_s$ for every $\theta \in \Theta_\Delta$.
\end{definition}

  \begin{definition}[Differential-difference polynomials]
 Let $\mathcal{R}$ be a $\Delta$-$\sigma$-ring.  
 The differential-difference polynomial ring over $\mathcal{R}$ in $\bm{y}=y_1,\ldots,y_n$
is defined as
 \[
 \mathcal{R}[\bm{y}_\infty] := \mathcal{R}[\theta\sigma^iy_s\mid \theta \in \Theta_\Delta;\,i\geqslant0;\,1\leqslant s\leqslant n].
 \]
 The structure of $\Delta$-$\sigma$ ring is defined by $\sigma(\theta\sigma^jy_s):=\theta\sigma^{j+1}y_s$ and $\partial_i(\theta\sigma^jy_s):=(\partial_i\theta)\sigma^jy_s$ for every $\theta \in \Theta_\Delta$ and $j \geqslant 0$.

 A $\Delta$-$\sigma$-polynomial is an element of $\mathcal{R}[\bm{y}_\infty]$. 
 Given $B\in\mathbb N$, let $\mathcal{R}[\bm{y}_B]$ denote the polynomial ring \[\mathcal{R}[\theta\sigma^jy_s\mid \theta\in\Theta_\Delta(B); 0\leqslant j\leqslant  B; 1\leqslant s\leqslant n].\]  
\end{definition}

The notions from logic that we use are described in 
detail in~\cite{Marker}.  In particular, we will use the notions of
a first-order language~\cite[Definition 1.1.1]{Marker}, 
structure~\cite[Definition 1.1.2]{Marker}, 
formula~\cite[Definition 1.1.5]{Marker},
theory~\cite[Section 1.2, page 14]{Marker}, 
model~\cite[Section 1.2, page 14]{Marker}, 
compactness~\cite[Section 2.1]{Marker},
complete theory~\cite[Definition 2.2.1]{Marker}, 
decidable theory~\cite[Definition 2.2.7]{Marker}, 
quantifier elimination~\cite[Definition 3.1.1]{Marker},
and $\aleph_0$-saturation~\cite[Definition 4.3.1]{Marker}.

\section{Main results}\label{sec:main_results}

For clarity, we gather our main results in one section.

\begin{theorem}[Upper bound for  irreducible components for PDEs] There exists a computable function $\components(m, n)$ such that, for every
differential field $k$  of zero characteristic with a set of $m$ commuting derivations $\Delta$ and finite $F\subset k\{y_1,\ldots,y_n\}_{\Delta}$ with $\max\{\ord F,\deg F\}\leqslant s$,
 the number of components in the variety defined by $F = 0$ does not exceed $\components(m, \max\{n,s\})$.
\end{theorem}
Additional details and proof are given in Theorem~\ref{thm:components}. 

\begin{theorem}[Upper bound for elimination in delay PDEs]For all non-negative integers $r$, $m$ and $s$, there exists a computable $B=B(r,m, s)$ such that, for all:  \begin{itemize}
\item non-negative integers $q$ and $t$,
   \item a $\Delta$-$\sigma$-field
    $k$ with $\Char k =0$ and $|\Delta|=m$,   
\item sets of $\Delta$-$\sigma$-polynomials $F\subset k[\bm{x}_{t},\bm{y}_{s}]$, where $\bm{x} =x_1,\ldots,x_q$,  $\bm{y}=y_1,\ldots,y_r$, and 
$\deg_{\bm{y}}F\leqslant s$, 
   \end{itemize}
   we have 
     \begin{multline*}\big\langle \sigma^i(F)\mid i\in \mathbb{Z}_{\geqslant 0} \big\rangle^{(\infty)}\cap k[\bm{x}_\infty]\ne\{0\}\\ \iff \langle \sigma^i(F)\mid i\in [0,B] \big\rangle^{(B)}\cap k[\bm{x}_{B+t}]\ne\{0\}.\end{multline*}
\end{theorem}

\begin{corollary}[Effective Nullstellensatz for delay PDEs]For all non-negative integers $r$, $m$ and $s$, there exists a computable $B=B(r,m, s)$ such that, for all:  \begin{itemize}
   \item  $\Delta$-$\sigma$-fields
    $k$ with $\Char k =0$ and $|\Delta|=m$,   
\item sets of $\Delta$-$\sigma$-polynomials $F\subset k[\bm{y}_{s}]$, where  $\bm{y}=y_1,\ldots,y_r$, and 
$\deg F\leqslant s$, 
   \end{itemize}
    the following statements are equivalent:
   \begin{enumerate}
  \item There exists a $\Delta$-$\sigma^*$ field $L$ extending $k$ such that $F=0$ has a sequence solution in $L$.
   \item $1 \notin \langle \sigma^i(F)\mid i\in [0,B] \big\rangle^{(B)}$.
   \item  There exists a field  extension $L$ of $k$ such that the polynomial system
     $\big\{\sigma^i(F)^{(j)}=0\mid i,j \in [0,B]\big\}$
      in the finitely many unknowns $\bm{y}_{B+s}$ has a solution in  $L$.
   \end{enumerate}
\end{corollary}

The two preceding theorems are proved using our main technical result about algorithms performing computations in complete 
 decidable theories.
Stating it precisely requires defining  admissible algorithms carefully, so we postpone it until Section~\ref{sec:bound_theories} and give here a simplified and informal version of the statement. 

\begin{theorem}[Algorithm yields a bound, stated precisely as Theorem~\ref{lem:size_bound}]There exists a computable function  with input
    \begin{itemize}
        \item 
        a complete decidable  theory $T$;
        \item an algorithm $\mathcal{A}$ performing computations in a model of $T$ restricted to using only definable functions when working with elements of the model;
        \item positive integer $\ell$
    \end{itemize}
    that computes a number $N$ such that for every model $M$ of $T$ and every $\mathbf{a} \in M^\ell$
    the size of the output of $\mathcal{A}$ with input $\mathbf{a}$ does not exceed $N$.
\end{theorem}

For the application of this to the Rosenfeld-Gr\"obner algorithm, see Theorem~\ref{thm:RG}.


\section{Bounds for the output size of algorithms over complete theories}\label{sec:bound_theories}

In this section, we will use the formalism of oracle Turing machines~\cite[\S~14.3]{Papadimitriou}.
Roughly speaking, an oracle Turing machine is a Turing machine with an extra tape for performing queries to an external oracle.
An oracle is not considered to be a part of the machine.

\subsection{Setup}\label{subsec:gen_algo}


To consider an algorithm dealing with elements of a (not necessarily computable)
model of a theory $T$, we will ``encapsulate'' the elements of the model given to
the algorithm into an oracle that allows to perform only first-order operations with 
them as defined below.
Alternatively, one could adapt other approaches used to formalize computations in real numbers~\cite[Section~3]{BCSS1998} or in arbitrary structures (see~\cite[\S 1]{Poizat} and~\cite[\S 2.2]{CK99}).

\begin{definition}[$T$-oracle]
  Let $\mathcal{L}$ be a language and $T$ be a theory in $\mathcal{L}$.
  For elements $a_1, \ldots, a_\ell$ of a model $M$ of $T$, any oracle  that supports the following queries: given a formula $\varphi(x_1, \ldots, x_\ell)$, the oracle returns the 
  value $\varphi(a_1, \ldots, a_\ell)$ in $M$ (can be true or false),  will be denoted by $\mathcal{O}_{M}(a_1, \ldots, a_\ell)$  and called an {\em evaluation oracle}.

\end{definition}

\begin{definition}[Total algorithm over $T$]\label{def:gen_alg}
  An oracle Turing machine $\mathcal{A}$ will be called \emph{a total algorithm over $T$} if,
  for all positive integers $\ell$, every model $M$ of $T$ and every $a_1, \ldots, a_\ell \in M$,  the machine with  every input and oracle $\mathcal{O}_M(a_1, \ldots, a_\ell)$ is guaranteed to terminate.
\end{definition}

\subsection{Auxiliary bound and result}
\begin{lemma}\label{lem:lemvarphi}
  There is an algorithm that takes as input:
  \begin{itemize}
      \item language $\mathcal{L}$;
      \item 
      a complete decidable theory $T$ given by a Turing machine checking correctness of sentences in the theory;
      \item a total algorithm $\mathcal{A}$ over $T$;
      \item positive integers $\ell$ and $N$;
      \item a string $\mathcal S$ in the input alphabet of $\mathcal{A}$;
  \end{itemize}
  and computes 
  \begin{itemize}
      \item a first-order formula $\varphi = \varphi_{T, \mathcal{A}}(\ell, \mathcal{S}, N)$ in $\mathcal{L}$ in $\ell$ variables and
      \item a number $\mathcal{N} := \mathcal{N}_{T, \mathcal{A}}(\ell, \mathcal{S}, N)$
  \end{itemize}
  such that, for any model $M$ of $T$ and
  tuple $\mathbf{a} \in M^\ell$, the following are equivalent:
  \begin{enumerate}
      \item the sentence $\varphi(\mathbf{a})$ is true in $M$;
      \item algorithm $\mathcal{A}$ with input $\mathcal{S}$ and oracle $\mathcal{O}_M(\mathbf{a})$ terminates after performing at most $N$ queries to the oracle
  \end{enumerate}
  and if these statements are true, then 
  the number of steps performed by $\mathcal{A}$ with input $\mathcal{S}$ and oracle $\mathcal{O}_M(\mathbf{a})$ does not exceed $\mathcal{N}$.
\end{lemma}

\begin{proof}
  We describe an algorithm for computing $\varphi_{T, \mathcal{A}}(\ell, \mathcal{S}, N)$ and $\mathcal{N}_{T, \mathcal{A}}(\ell, \mathcal{S}, N)$.
  Fix some $\mathcal{L}, T, \mathcal{A},$ $\ell$, and $\mathcal{S}$.

  We will describe an algorithm that, for a given positive integer $s$, computes first-order formulas $\psi_s$ and $q_s$ in $\mathcal{L}$ in the variables $\mathbf{x} = (x_1, \ldots, x_\ell)$ and a positive integer $\mathcal{N}_s$ such that, for every model $M$ of $T$ and every $\mathbf{a} \in T^\ell$
  \begin{itemize}
      \item $\psi_s(\mathbf{a})$ is true in $M$ iff algorithm $\mathcal{A}$ with input $\mathcal{S}$ and oracle $\mathcal{O}_M(\mathbf{a})$ will perform at least $s$ queries;
      \item if $\psi_s(\mathbf{a})$ is true in $M$, then the result of the $s$-th query will be $q_s(\mathbf{a})$;
      \item if algorithm $\mathcal{A}$ with input $\mathcal{S}$ and oracle $\mathcal{O}_M(\mathbf{a})$ performs at most $s$ queries, then 
     the number of steps performed does not exceed $\mathcal{N}_s$.
  \end{itemize}
  
  Fix some $s \geqslant 1$ and assume that the algorithm have computed $\psi_1, \ldots, \psi_{s - 1}$, $q_1, \ldots, q_{s - 1}$, and $\mathcal{N}_{0}, \ldots, \mathcal{N}_{s - 2}$. 
  Assume that $\mathcal{A}$ with input $\mathcal{S}$ has performed $s - 1$ queries.
  Then whether or not an $s$-th query will be performed is determined by the results of the first $s - 1$ queries.
  Fix some $\mathbf{r} \in \{\mathrm{True}, \mathrm{False}\}^{s - 1}$.
  It will represent possible results of the first $s - 1$ queries.
  Consider the following formula in $\mathcal{L}$:
  \[
    \psi_{\mathbf{r}}(\mathbf{x}) := \psi_{s - 1}(\mathbf{x}) \wedge \bigwedge\limits_{i = 1}^{s - 1} \left( q_i(\mathbf{x}) \iff r_i \right),
  \]
  where we assume $\psi_0 = \mathrm{True}$.
  The algorithm uses the algorithm for checking correctness of sentences in $T$
  to check whether the sentence $\exists\mathbf{x}\; \psi_{\mathbf{r}}(\mathbf{x})$ is false in~$T$.
  If it is, then there is no oracle of the form $\mathcal{O}_M(\mathbf{a})$ such that $\mathcal{A}$ will perform at least $s - 1$ queries on it with the results being $r_1, \ldots, r_{s - 1}$.
  
  In the case of $\exists \mathbf{x}\; \psi_{\mathbf{r}}(\mathbf{x})$ is true in $T$,
  the algorithm will run $\mathcal{A}$ with input $\mathcal{S}$ and an oracle $\mathcal{O}_{\mathbf{r}}$ that works as follows.
  For the first $s - 1$ queries, $\mathcal{O}_{\mathbf{r}}$ will return $r_1, \ldots, r_{s - 1}$.
  For all subsequent queries, it always returns True.
  The algorithm will stop the execution of $\mathcal{A}$ if $\mathcal{A}$ makes an $s$-th query to the oracle, and denote the formula in the query by $q_{\mathbf{r}}$.
  
  Since $\exists\mathbf{x}\;\psi_{\mathbf{r}}(\mathbf{x})$ is true in $T$, $\mathcal{O}_{\mathbf{r}}$ gives the same responses to the first $s - 1$ queries as some oracle of the form $\mathcal{O}_M(\mathbf{a})$.
  Since $\mathcal{A}$ must terminate in finite time for every such oracle, one of the following  must happen:
  \begin{enumerate}
      \item $\mathcal{A}$ will perform an $s$-th query. 
      \item $\mathcal{A}$ will terminate after performing only $s - 1$ queries.
  \end{enumerate}
  In the former case, as described above, the algorithm will define a formula $q_{\mathbf{r}}$ to be the $s$-th query.
  In the latter case, the algorithm will define $\mathcal{N}_{\mathbf{r}}$ to be 
  the number of steps  performed by $\mathcal{A}$.
  Then the algorithm computes
\begin{gather*}
      \psi_s(\mathbf{x}) := \bigvee\limits_{q_{\mathbf{r}} \text{ is defined}} \psi_{\mathbf{r}}(\mathbf{x}),\quad
      q_s(\mathbf{x}) := \bigwedge\limits_{q_{\mathbf{r}}\text{ is defined}} (\psi_{\mathbf{r}}(\mathbf{x}) \implies q_{\mathbf{r}}(\mathbf{x})),\\
\mathcal{N}_{s - 1} := \max\left(\mathcal{N}_{s - 2}, \sum\limits_{\mathcal{N}_{\mathbf{r}} \text{ is defined}} \mathcal{N}_{\mathbf{r}}\right),
\end{gather*}
   where we assume $\mathcal{N}_{-1} = -\infty$.
   If the set $\{\mathbf{r}\mid q_{\mathbf{r}} \text{ is defined}\}$ is empty, the algorithm sets $\psi_s(\mathbf{x}) = \operatorname{False}$ and $q_s(\mathbf{x}) = \operatorname{True}$.
   Finally, the algorithm returns $\varphi_{T, \mathcal{A}}(\ell, \mathcal{S}, N) := \lnot \psi_{N + 1}$ and $\mathcal{N}_{T, \mathcal{A}}(\ell, \mathcal{S}, N) := \mathcal{N}_N$.
\end{proof}


\begin{lemma}\label{lem:baer_category}
  Let $T$ be a theory and $M$ an $\aleph_0$-saturated model.
  Let $U_1 \supset U_2 \supset U_3 \supset \ldots$ be a sequence of definable sets in $M^n$ such that
  $\bigcap\limits_{i = 1}^\infty U_i = \varnothing$. Then there exists $N$ such that $U_N = \varnothing$.
\end{lemma}

\begin{proof}
  Assume the contrary, that is, that $U_i \neq \varnothing$ for every $i \geqslant 1$. 
  We will show that $\bigcap\limits_{i = 1}^\infty U_i \neq \varnothing$.

  We show that a collection of formulas $\{x \in U_i\}_{i = 1}^\infty$ is finitely satisfiable. 
  Indeed, let $S \subset \mathbb{Z}_{> 0}$ be a finite set and $N = \max S$.
  Then $\bigcap_{i \in S} U_i = U_N \neq \varnothing$.
  Due to compactness, the countable collection $\{x \in U_i\}_{i = 1}^\infty$ is satisfiable in some elementary extension of $M$.
  Since $M$ is $\aleph_0$-saturated, this collection is satisfiable in $M$.
  Therefore, $\bigcap\limits_{i = 1}^\infty U_i \neq \varnothing$.
\end{proof}

\subsection{Main result}
\begin{theorem}\label{lem:size_bound}
    There exists a computable function $\size_{T, \mathcal{A}}(\ell, r)$ with input
    \begin{itemize}
        \item 
        a complete decidable theory $T$ (given by an algorithm for checking correctness of sentences);
        \item a total algorithm $\mathcal{A}$ over $T$;
        \item positive integers $\ell$ and $r$
    \end{itemize}
    that computes a number $N$ such that for every model $M$ of $T$, every $\mathbf{a} \in M^\ell$, and every string $\mathcal{S}$ in the alphabet of $\mathcal{A}$ of size at most $r$,
     the number of steps performed by $\mathcal{A}$ with input $\mathcal{S}$ and oracle $\mathcal{O}_M(\mathbf{a})$ does not exceed $N$.
\end{theorem}

\begin{remark}\label{rem:bitsize}
    Let the intermediate result at step $n$ for a total algorithm $\mathcal{A}$ with given input and oracle be the content of all the cells of the tape that have been read by the Turing machine. 
    Since a Turing machine can read at most one cell at each step, the number of these cells cannot exceed $n$.
    Therefore, the intermediate result at step $n$ can be encoded using $n \log \ell$ bits, where $\ell$ is the cardinality of the alphabet of $\mathcal{A}$.
    In particular, if a binary alphabet is used, the bitsize of the intermediate result never exceeds the total number of steps in the algorithm.
\end{remark}

\begin{proof}
   We will describe an algorithm for computing $\size_{T, \mathcal{A}}(\ell, r)$.
   We fix $T$, $\mathcal{A}$, $\ell$, and $r$.
   We will consider $\mathcal{S}$ of length at most $r$ and describe how to compute a bound for  
   the number of steps given that the input is $\mathcal{S}$.
   Taking the maximum over all $\mathcal{S}$ of length at most $r$ (there are finitely many of them), we obtain $\size_{\mathcal{A}, T}(\ell, r)$.
   
   The algorithm will compute $\varphi_i := \varphi_{T, \mathcal{A}}(\ell, \mathcal{S}, i)$ for $i = 1, 2, \ldots$ using the algorithm from  Lemma~\ref{lem:lemvarphi}.
   For each $\varphi_i$, the algorithm will check whether the formula is equivalent to $\mathrm{True}$ in $T$ using the decidability of $T$.

   If this is true, the algorithm stops and returns $\mathcal{N}_{T, \mathcal{A}}(\ell, \mathcal{S}, i)$ (see Lemma~\ref{lem:lemvarphi}).
   It remains to show that the described procedure terminates in finitely many steps.
   Let $M$ be an $\aleph_0$-saturated model of $T$ {(it exists, for example, due to~\citep[Theorem~4.3.12]{Marker})}.
   For every $i = 1, 2, \ldots$, we introduce a definable set
   \[
     U_i := \{\mathbf{a} \in M^\ell \mid \varphi_i(\mathbf{a}) = \mathrm{False}\}.
   \]
   Notice that $U_i = \varnothing$ if and only if $(\varphi_i \iff \mathrm{True})$ in $T$.
   Then the definition of $\varphi_i$'s implies that $U_1 \supset U_2 \supset \ldots$.
   Assume that $\bigcap_{i = 1}^\infty U_i$ is not empty and choose an element $\mathbf{a}$ in it.
   Then $\mathcal{A}$ will not terminate in finitely many steps with input $\mathcal{S}$ and oracle $\mathcal{O}_M(\mathbf{a})$.
   Thus, $\bigcap_{i = 1}^\infty U_i = \varnothing$.
   Lemma~\ref{lem:baer_category} implies that there exists $N$ such that $U_N = \varnothing$.
   Then our algorithm will terminate after 
   checking whether~$\varphi_N$  is equivalent to True.
\end{proof}


\section{Applications to differential algebra}\label{sec:diffalg}

{In this section, we will apply the results of Section~\ref{sec:bound_theories} to the theory of differentially closed fields with several commuting derivations.}

\subsection{Preparation}

\begin{notation}\label{not:dcf}
  Let $m$ be a positive integer.
  \begin{itemize}
      \item The language of partial differential rings with $m$ commuting derivation is denoted by
     $
      \mathcal{L}_{m} := \{+, -, \cdot, 0, 1, \partial_1, \ldots, \partial_m\}$.
      {We add a separate functional symbol for subtraction for convenience.}
      \item The theory of partial differentially closed fields with $m$ commuting derivations of characteristic zero is denoted by $\operatorname{DCF}_{m}$.
      Recall that $\operatorname{DCF}_m$ is complete~\cite[Corollary~3.1.9
      ]{mcgrail}  and, with this, is 
      decidable by~\cite[Lemma~2.2.8]{Marker} and~\cite[Lemma~3.1.2 and page~890]{mcgrail}.
  \end{itemize}
\end{notation}

\begin{notation}\label{not:pols}
  Let $m, n, h$ be positive integers and $k$ a differential field with a set of $m$ commuting derivations $\Delta=\{\partial_1, \ldots, \partial_m\}$.
  \begin{itemize}
      \item $\pol_k(m, n, h)$ denotes the space of all differential polynomials over $k$ in $n$ variables of order at most $h$ and degree at most $h$.
      \item The dimension of $\pol_k(m, n, h)$ (which does not depend on $k$) will be denoted by $\poldim(m, n, h)$.
  \end{itemize}
\end{notation}

\begin{notation}\label{not:Lm_polys}
Let $m$, $\ell$ and $n$ be positive integers.
\begin{itemize}
    \item Let $\mathcal{L}_m(x_1, \ldots, x_\ell)\{y_1, \ldots, y_n\}_\Delta$  denote the ring of differential polynomials in differential variables $y_1, \ldots, y_n$ with respect to $m$ derivations with the coefficients being terms in the language $\mathcal{L}_m$ in $x_1, \ldots, x_\ell$ (that is, elements of $\mathbb{Z}\{x_1, \ldots, x_\ell\}_\Delta)$. 
       
    This is a computable differential ring with $m$ commuting derivations. 
    In what follows, we will assume that the algorithms use dense representation to store these polynomials (that is, store all the coefficients up to certain order and certain degree).
    
    \item Let $k$ be a differential field with $m$ derivations and $\mathbf{a} \in k^\ell$.
    Then, for $T \in \mathcal{L}_m(x_1, \ldots, x_\ell)\{y_1, \ldots, y_n\}_\Delta$, we define $T(\mathbf{a}) \in k\{y_1, \ldots, y_n\}_\Delta$ to be the result of evaluating the coefficients of $T$ at $\mathbf{a}$.
\end{itemize}
\end{notation}

\begin{definition}\label{def:diffranking}
  A {\em differential ranking}  for $k\{z_1,\ldots,z_n\}_\Delta$ is a total order $>$ on $Z := \{\theta z_i\mid \theta\in \Theta_\Delta,\, 1\leqslant i\leqslant n\}$ satisfying, for all $i$, $1\leqslant i\leqslant m$:
  \begin{itemize}
    \item for all $x \in Z$, $\partial_i(x) > x$ and
    \item for all $x, y \in Z$, if $x >y$, then $\partial_i(x) > \partial_i(y)$.
  \end{itemize}
\end{definition}

\begin{notation}\label{not:rank}
For a $\Delta$-field $k$ and
  $f \in k\{z_1,\ldots,z_n\}_\Delta \backslash k$ and differential ranking $>$,
  \begin{itemize}
    \item $\lead(f)$ is the element of $Z$ of the highest rank appearing in $f$. 
    \item The leading coefficient of $f$ considered as a polynomial in $\lead(f)$ is denoted by $\init(f)$ and called the initial of $f$. 
    \item The separant of $f$ is $\frac{\partial f}{\partial\lead(f)}$.
    \item The rank of $f$ is $\rank(f) = \lead(f)^{\deg_{\lead(f)}f}$. 
    The ranks are compared first with respect to $\lead$, and in the case of equality with respect to $\deg$.
    \item For $S \subset k\{z_1,\ldots,z_n\}_\Delta \backslash k$, the set of initials and separants of $S$ is denoted by $H_S$.
  \end{itemize}
\end{notation}

\begin{remark}[Defining a ranking]\label{rem:def_rank}
  In general, there are uncountable many differential rankings already for $m = 2$ and $n = 1$.
  However, \cite[Theorem~29]{rankings} implies that any differential ranking can be defined by $m(m + 1)n$ real numbers together with $n^2$ integers not exceeding $m$ and one permutation on $n$ elements.
  We define a function $\RK_{m, n}(\bm{\alpha}, \mathcal{S})$ taking as input a tuple $\bm{\alpha}$ of $m(m + 1)n$ real numbers and a binary string $\mathcal{S}$ (of length at most $(n^2 + n)\log_2(\max(n, m))$) encoding the integers and the permutation and returning the corresponding binary predicate on the derivatives as in~\cite[Definition~28]{rankings}.
  The relevant properties of this encoding for us will be that, for fixed $\mathcal{S}$:
  \begin{enumerate}
      \item the statement that $\RK_{m, n}(\bm{\alpha}, \mathcal{S})$ defines a ranking is a first-order formula in $\bm{\alpha}$ in the language of ordered fields;
      \item for every two derivatives $\theta_1 z_i$ and $\theta_2 z_j$, the fact that $\theta_1 z_i < \theta_2 z_j$ with respect to $\RK_{m, n}(\bm{\alpha}, \mathcal{S})$ is also a first-order formula in $\bm{\alpha}$ in the language of ordered fields.
  \end{enumerate}
\end{remark}

\begin{definition}[Characteristic sets]
  \begin{itemize}
  \item[]
    \item For $f, g \in k\{z_1,\ldots,z_n\}_\Delta \backslash k$, $f$ is said to be reduced w.r.t. $g$ if no proper derivative of $\lead(g)$ appears in $f$ and $\deg_{\lead(g)}f <\deg_{\lead(g)}g$.
    \item 
    A subset $\mathcal{A}\subset k\{z_1,\ldots,z_n\}_\Delta \backslash k$
    is called {\em autoreduced} if, for all $p \in \mathcal{A}$, $p$ is reduced w.r.t. every  element of $\mathcal A\setminus \{p\}$. 
    One can show that every autoreduced set is finite \cite[Section~I.9]{Kol}. 
    \item Let $\mathcal{A}=A_1<\ldots<A_r$ and $\mathcal{B} = B_1<\ldots<B_s$ be autoreduced sets ordered by their ranks (see Notation~\ref{not:rank}). We say that $\mathcal{A} < \mathcal{B}$ if
    \begin{itemize}
      \item $r > s$ and $\rank(A_i)=\rank(B_i)$, $1\leqslant i\leqslant s$, or
      \item there exists $q$ such that $\rank(A_q) <\rank(B_q)$ and, for all $i$, $1\leqslant i< q$, $\rank(A_i)=\rank(B_i)$.
    \end{itemize}
    \item An autoreduced subset of the smallest rank of a differential ideal $I\subset k\{z_1,\ldots,z_n\}_\Delta$
    is called a {\em characteristic set} of $I$. One can show that every non-zero differential ideal in $k\{z_1,\ldots,z_n\}_\Delta$ has a characteristic set.
    \item A radical differential ideal $I$ of $k\{z_1,\ldots,z_n\}_\Delta$ is said to be {\em characterizable} if $I$ has a characteristic set $C$ such that $I=\langle C\rangle^{(\infty)}:H_C^\infty$.
  \end{itemize}
\end{definition}

 The Rosenfeld-Gr\"obner algorithm~\cite[Theorem~9]{BLOP09} takes as input
a finite set $F$ of differential polynomials and a differential ranking and outputs autoreduced sets $\mathcal{C}_1,\ldots,\mathcal{C}_N$ such that
\begin{equation}\label{eq:RG}
\sqrt{\langle F\rangle^{(\infty)}} = \bigcap_{i=1}^N\langle \mathcal{C}_i\rangle^{(\infty)}:H_{\mathcal{C}_i}^\infty
\end{equation}
and that, for each $i$, $1\leqslant i \leqslant N$, $\mathcal{C}_i$ is a characteristic set of $\langle \mathcal{C}_i\rangle^{(\infty)}:H_{\mathcal{C}_i}^\infty$. The representation~\eqref{eq:RG} can be used, for example, for membership testing,  
 estimating
the number of irreducible components (used in Theorem~\ref{thm:components}) or the Kolchin polynomial (used in Section~\ref{sec:delayPDE}) of a differential-algebraic variety.


 With the next Proposition~\ref{prop:rg_gen} we express how we will call the Rosenfeld-Gr\"{o}bner algorithm.   
This algorithm depends on the choice of a differential ranking.  
The reader may wish to make one such choice
once and for all, thereby ignoring the potential ambiguity. However, since such a choice may affect the size of the output and the efficiency of 
any given implementation of the algorithm, one may prefer to allow for these other orderings. 

We will express this dependence by seeing the algorithm as a total algorithm relative to the two-sorted theory $\DCF_m\oplus \RCF$ which is a disjoint union of $\DCF_m$ and the  complete decidable theory with quantifier elimination of real closed fields $\RCF$~\cite[Theorem 3.3.15 and Corollary 3.3.16]{Marker}.
Then we will use the characterization of differential rankings via real numbers from Remark~\ref{rem:def_rank}.

\begin{lemma}
  Theory $\DCF_m\oplus \RCF$ is  decidable and complete.
\end{lemma}

\begin{proof}
  In order to prove the completeness and decidability, we will prove that there is an algorithm for quantifier elimination in 
  $\DCF_m\oplus \RCF$ based on the
  existence of such algorithms for 
  $\DCF_m$ (follows from decidability, see Notation~\ref{not:dcf}, and quantifier elimination~\cite[Theorem 3.1.7]{mcgrail}) and $\RCF$.
  It is sufficient to perform 
  quantifier elimination for a formula of the form
  \[
    \exists x \in S\colon L_1\wedge \ldots \wedge L_N,
  \]
  where $S$ is one of the sorts (corresponding to $\DCF_m$ or $\RCF$) and $L_1, \ldots, L_N$ are literals. (See~\cite[Lemma 3.1.5]{Marker}.) 
  By reordering $L_1, \ldots, L_N$ if necessary, we will further assume that there exists $N_0$ such that $L_1, \ldots, L_{N_0}$ are in the signature of the sort $S$ and $L_{N_0 + 1}, \ldots, L_N$ are in the signature of the other sort.
  Then 
  \[
    \left(\exists x \in S\colon L_1\wedge \ldots L_N\right) \iff \left(\exists x \in S \colon L_1\wedge \ldots \wedge L_{N_0}\right) \wedge \left( L_{N_0 + 1}\wedge \ldots \wedge L_{N}\right),
  \]
  and, for $\exists x \in S \colon L_1\wedge \ldots \wedge L_{N_0}$, the algorithm for the corresponding sort $S$ can compute an equivalent quantifier-free formula.

The resulting theory is decidable because the correctness of each sentence can be checked by performing quantifier elimination after which the formula will become just true/false.
\end{proof}

\begin{proposition}\label{prop:rg_gen}
  There is a computable function that, for a given  
  positive integer $m$,
  computes a total algorithm $\mathcal{RG}_{m,}$ over 
 $\DCF_m \oplus \RCF$
  such that, for every differential field $k$ with $m$ derivations and $\mathbf{a} \in k^\ell$ and  any 
  $\mathbf{b} \in \mathbb{R}^s$, the input-output specification of $\mathcal{RG}_{m}$ with oracle $\mathcal{O}_{k \oplus \mathbb{R}}(\mathbf{a}, \mathbf{b})$ is the following:
  \begin{description}
    \item[Input] finite subsets $A$ and $S$ of $\mathcal{L}_m(x_1, \ldots, x_\ell)\{y_1, \ldots, y_n\}_\Delta$ and a binary string $\mathcal{S}$;
    \item[Output] if $\RK_{m, n}(\mathbf{b}, \mathcal{S})$ (see Remark~\ref{rem:def_rank}) defines a differential ranking, return a list of tuples $C_1, \ldots, C_N$ from $\mathcal{L}_m(x_1, \ldots, x_\ell)\{y_1, \ldots, y_n\}_\Delta$  such that 
    \[
    C_1(\mathbf{a}), \ldots, C_N(\mathbf{a})
    \] 
    is the output of the Rosenfeld-Gr\"obner algorithm~\cite[Theorem~9]{BLOP09} with input $(A(\mathbf{a}), S(\mathbf{a}))$  with respect to the ranking $\RK_{m, n}(\mathbf{b}, \mathcal{S})$. Otherwise, return $\varnothing$.
  \end{description}
\end{proposition}
  
\begin{proof}
   \cite[Theorem~9]{BLOP09} states that the only operations performed by the Rosenfeld-Gr\"obner algorithm with the elements of the ground differential field are arithmetic operations, differentiation, and zero testing.
   {Algorithm $\mathcal{RG}_{m}$ is constructed to work exactly} in the same way as the Rosenfeld-Gr\"obner algorithm with the only difference that  
   the elements of the ground differential field will be represented as $L(\mathbf{a})$, where $L \in \mathcal{L}_m(x_1, \ldots, x_\ell)\{y_1, \ldots, y_n\}_\Delta$.
   The arithmetic operations and differentiations can be performed with $L$, zero testing can be performed using the $k$-component of the oracle, and the queries to the ranking can be performed using the $\mathbb{R}$-component of the oracle, 
   so $\mathcal{RG}$ will be able to perform the same computations as the Rosenfeld-Gr\"obner algorithm.
   
   Due to~\cite[Theorem~5]{BLOP09},
   the Rosenfeld-Gr\"obner algorithm is guaranteed to terminate on every input.
   Hence, the same is true for $\mathcal{RG}_{m}$.
\end{proof}


\subsection{Bounds}
\begin{theorem}[Upper bound for Rosenfeld-Gr\"obner algorithm]\label{thm:RG}
    There exists a computable function $\operatorname{RG}(m, n, \ell)$ such that, for every
    differential field $k$ with $m$ derivations and subsets
    $A, S \subset \pol_k(m, n, n)$ with $|A|, |S| \leqslant \ell$, and every differential ranking,
    the Rosenfeld-Gr\"obner algorithm~\cite[Theorem~9]{BLOP09} on $A$ and $S$ will produce at most $\operatorname{RG}(m, n, \ell)$ components with all the orders and degrees  of the differential polynomials occurring in the algorithm not exceeding $\operatorname{RG}(m, n, \ell)$.
\end{theorem}

\begin{proof}
   We fix integers $m$, $n$, and $\ell$ and compute the total algorithm $\mathcal{RG}_{m}$ over $\DCF_{m} \oplus \RCF$ 
   from Proposition~\ref{prop:rg_gen}.
   Let $\mathbf{a}$ be the set of all the coefficients of $A$ and $S$.
   Then $|\mathbf{a}| \leqslant N := 2\ell \poldim(m, n, n)$.
   The sets $A$ and $S$ can be presented as evaluations of subsets $\widetilde{A}, \widetilde{S} \subset \mathcal{L}_m(x_1, \ldots, x_N)\{y_1, \ldots, y_n\}_\Delta$ at $\mathbf{a}$ such that the orders and degrees of $\widetilde{A}, \widetilde{S}$ in $y_1, \ldots, y_n$ do not exceed $n$ and every coefficient is a single variable $x_i$.
   Let the ranking be defined as $\RK(\mathbf{b}, \mathcal{S})$ (see Remark~\ref{rem:def_rank}), where $\mathbf{b}$ is a tuple of $m(m + 1)n$ real numbers and $\mathcal{S}$ is a binary string of length at most $(n^2 + n) \log_2 \max(n, m)$.
   Then the the tuple $(\widetilde{A}, \widetilde{S}, \mathcal{S})$ can be encoded as a binary string of the length bounded by a computable function $S(m, n, N)$.
   
   We run $\mathcal{RG}_{m}$ with the input $\mathcal{I} = (\widetilde{A}, \widetilde{S}, \mathcal{S})$ and oracle $\mathcal{O}(\mathbf{a}, \mathbf{b})$.
   Theorem~\ref{lem:size_bound} implies that the number of steps and, consequently, the bitsize of of all the intermediate results (see Remark~\ref{rem:bitsize})
   will not exceed $\size_{\mathcal{RG}_{m}, \DCF_m \oplus \RCF}(N, S(m, n, N))$.
   
   Since each component takes at least one bit, a polynomial of degree $d$ or order $d$ has at least $d$ coefficients (due to the dense representation of the polynomials, see Notation~\ref{not:pols}) requiring at least one bit each, the number of components, the degrees and orders do not exceed the bitsize of the intermediate results. 
   Therefore, we can set $\operatorname{RG}(m, n, \ell) = \size_{\mathcal{RG}_{m}, \DCF_{m} \oplus \RCF}(N, S(m, n, N))$.
\end{proof}

\begin{corollary}\label{cor:charset}
   There exists a computable function $\operatorname{CharSet}(m, n, \ell)$ such that, for every
    computable differential field $k$ with $m$ derivations and subsets
    $A, S \subset \pol_k(m, n, n)$ with $|A|, |S| \leqslant \ell$, and every differential ranking,
    the ideal  $\sqrt{\langle A\rangle^{(\infty)} \colon S^{\infty}}$ can be written as an intersection of at most $\operatorname{CharSet}(m, n, \ell)$ characterizable differential ideals defined by their characteristic sets with respect to the ranking of order and degree not exceeding $\operatorname{CharSet}(m, n, \ell)$.
\end{corollary}

\begin{proof}
 Theorem~\ref{thm:RG} implies that there exists a representation
 \[
 \sqrt{\langle A\rangle^{(\infty)} \colon S^\infty} = (\langle C_1\rangle^{(\infty)} \colon H_{C_1}) \cap\ldots\cap (\langle C_N\rangle^{(\infty)} \colon H_{C_N}),
 \]
 where $H_{C_i}$ is the product of the initials and separants of $C_i$, and $C_i$ is the characteristic presentation~\cite[Definition~8]{BLOP09} of $\langle C_i\rangle^{(\infty)} \colon H_{C_i}^\infty$ for every $1 \leqslant i \leqslant N$.
 As noted in~\cite[p.~108]{BLOP09} a characteristic set of $\langle C_i\rangle^{(\infty)} \colon H_{C_i}^\infty$ can be obtained from $C_i$ by performing reductions until it will become autoreduced. 
 Since differential reduction is a part of the Rosenfeld-Gr\"obner algorithm, it can also be performed by a total algorithm over $\DCF_m \oplus \RCF$.
 Therefore, as in the proof of Theorem~\ref{thm:RG}, Lemma~\ref{lem:size_bound} implies that $\langle C_i\rangle^{(\infty)} \colon H_{C_i}^\infty$ has a characteristic set with degrees and order bounded by a computable function of the degrees and orders of $C_i$.
 The latter are bounded by a computable function $\operatorname{RG}$ due to Theorem~\ref{thm:RG}.
 Composing these two bounds, we obtain a desired function~$\operatorname{CharSet}(m, n, \ell)$.
\end{proof}


\begin{lemma}\label{lem:charset_components}
  There exists a computable function $\operatorname{PrimeComp}(m, n)$ such that for every partial differential field $k$ with $m$ derivations, every ranking, and every  characterizable differential ideal $I$ defined by a characteristic set $C \subset \pol_k(m, n, n)$ with respect to this ranking, we have
  \begin{enumerate}
      \item the number of prime components of $I$ does not exceed $\operatorname{PrimeComp}(m, n)$;
      \item every prime component of $I$ has a characteristic set with respect to the ranking with orders and degrees bounded by $\operatorname{PrimeComp}(m, n)$.
  \end{enumerate}
\end{lemma}

\begin{proof}
  Let $H$ be the product of the initials and separants of $C$.
  \cite[Theorem~4]{BLOP09}
 implies that the number of prime components of $\langle C\rangle^{(\infty)} \colon H^\infty$ is equal to the number of prime components of the algebraic ideal $(\langle C\rangle^{(\infty)} \colon H^\infty) \cap R_n$, where $R_n$ is the ring of differential polynomials of order at most $n$.
 Since the degrees of elements of $C$ are bounded by $n$, the B\'ezout inequality implies that there is a computable bound $D$ for the degree of  the radical ideal $I \cap R_n$ (defined, e.g., as the degree of the corresponding affine variety \cite[page~246]{Heintz}) in terms of $m$ and $n$, so this gives a bound for the number of components.
 
 Let $P_1, \ldots, P_\ell$ be the prime components of $I$.
 For every $1 \leqslant i \leqslant \ell$, $P_i \cap R_n$ is a prime algebraic ideal, and its zero set can be defined by equations of degree at most $\deg (P_i\cap R_n)$ due to~\cite[Proposition~3]{Heintz}.
 Therefore, for each $2 \leqslant i \leqslant \ell$, we can choose a polynomial in $(P_1 \setminus P_i) \cap R_n$ of degree at most $\deg (P_i \cap R_n)$.
 Their product $Q$ has degree at most $\deg (I\cap R_n) \leqslant D$.
 Observe that 
 \[
   P_1 = P_1 \colon Q^\infty \subset I \colon Q^\infty = (P_1 \colon Q^\infty) \cap \ldots \cap (P_\ell \colon Q^\infty) = P_1.
 \]
 Thus, applying Corollary~\ref{cor:charset} to a pair $(C, HQ)$ and using that $|C| \leqslant \poldim(m, n)$, we show that $P_1$ has a characteristic set with orders and degrees bounded by $\operatorname{CharSet}(m, D + n, \poldim(m, n))$. 
\end{proof}


\begin{theorem}[Upper bound for the components of a differential variety and their number]\label{thm:components}
    There exists a computable function $\components(m, n)$  such that, for  {all non-negative integers $m$, $n$ and $h$ and}
a partial differential field $k$ with $m$ derivations and finite set $F \subset \pol_k(m, n, h)$:
    \begin{enumerate}
        \item\label{bound:num_comp} the number of components in the variety defined by $F = 0$ does not exceed $\components(m, \max\{n,h\})$;
        \item\label{bound:char_set} for every differential ranking and every component $X$ of the variety $F = 0$, $X$ has a characteristic set with respect to the ranking with orders and degrees bounded by $\components(m, \max\{n,h\})$.
    \end{enumerate}
\end{theorem}

\begin{proof}
   Consider any differential ranking.
   By replacing $F$ with the basis of its linear span, we will further assume that $|F| \leqslant \poldim(m, n,  h)$ (see Notation~\ref{not:pols}).
   Corollary~\ref{cor:charset} implies that $\sqrt{\langle F \rangle^{(\infty)}}$ can be represented as an intersection of  at most $N$ characterizable ideals with characteristic sets $C_1, \ldots, C_N$ of order and degree at most $N$, where
   \[
   N := \operatorname{CharSet}(m,  \max\{n,h\}, \poldim(m, n, h)).
   \]
   Lemma~\ref{lem:charset_components} applied to each of $C_1, \ldots, C_N$ implies that the number of components of the variety defined by $F = 0$ does not exceed $N\cdot\operatorname{PrimeComp}(m, N)$, and each of them has a characteristic set with orders and degrees not exceeding $\operatorname{PrimeComp}(m, N)$.
\end{proof}

\begin{remark}
  It was shown in~\cite[Theorem~6.1]{HTKM12} that there exists a (not necessarily computable) bound for the degrees and orders a characteristic set of a prime differential ideal. The second part of Theorem~\ref{thm:components} implies that there is a computable bound.
\end{remark}

\section{Application to delay PDEs}\label{sec:delayPDE}
In this section, we will show how Theorem~\ref{thm:components} applies to the problem of elimination of unknowns in delay PDEs.

The proof of the main result of this section, Theorem~\ref{thm:main3} (Effective elimination theorem for delay PDEs) inherited from~\cite{ordinary_case} 
had only two missing ingredients closely related to each other: the bound on the number of components of the variety defined by a system of differential algebraic PDEs and bounds on the coefficients of Kolchin polynomials  under projection in the PDE case.
Now that we have obtained 
the former
in Theorem~\ref{thm:components} together with a bound for characteristic sets, it is possible to obtain the latter  in Lemma~\ref{lem:Kolchin_poly_projection} and finish the proof. 
Therefore, Section~\ref{sec:delayPDE} can be thought of as a motivation for the rest of the paper and is an interesting example of a problem from differential-difference algebra that motivated a purely differential algebraic result.

\subsection{Bounds for Kolchin polynomials for algebraic PDEs}

\begin{definition}
Let $K$ be a differentially closed $\Delta$-field containing a $\Delta$-field $k$. We say that $X \subset K^n$ is a {\em $\Delta$-variety} over $k$ if there exists $F \subset k\{y_1,\ldots,y_n\}_\Delta$ such that
\[
X =\{\bm{a}\in K^n\mid \forall\,f \in F\  f(\bm{a})=0\}.
\]
We write $X = \mathbb{V}(F)$. The $\Delta$-variety $K^n$ is denoted by $\mathbb{A}^n$. A $\Delta$-variety $\mathbb{V}(F)$ is called {\em irreducible} if the differential ideal $\sqrt{\langle F\rangle^{(\infty)}}$ is prime.

 For a subset $Y \subset K^n$, the smallest $\Delta$-variety $X \subset K^n$ containing $Y$ is called the {\em Kolchin closure} of $Y$ and denoted by $\overline{Y}^{\operatorname{Kol}}$.
\end{definition}

\begin{definition}
  We will say that a $\Delta$-variety $X \subset \mathbb{A}^n$ is \emph{bounded by $N$} if $N \geqslant \max(n, m)$ ($m=|\Delta|$)  and $X$ can be defined by equations of order and degree at most $N$.
\end{definition}
  
\begin{notation}
 For a numeric polynomial $\omega(t) = \sum\limits_{i = 0}^{m} a_i \binom{t + i}{i}$, we set 
 \[
 |\omega| := \sum\limits_{i = 0}^{m} |a_i|.
 \]
\end{notation}
  
\begin{definition}
The {\em generic point} $(a_1,\ldots,a_n)$ of  an irreducible $\Delta$-variety $X=\mathbb{V}(F)$, where $F\subset k\{\bm{y}\}_\Delta$, is the image of $\bm{y}$ under the homomorphism $K\{\bm{y}\}_\Delta\to  K\{\bm{y}\}_\Delta\big/\sqrt{\langle F\rangle^{(\infty)}}$.
\end{definition}

\begin{definition}
The {\em Kolchin polynomial} of an irreducible $\Delta$-variety $V=\mathbb{V}(F)$, where $F \subset k\{\bm{y}\}_\Delta$, 
 is the unique numerical polynomial $\omega_V(t)$  such that there exists $t_0\geqslant 0$ such that, for all $t \geqslant t_0$ and
 the
 generic point  
$\bm{a}$
 of $V$,  $\omega_V(t)=\trdeg k(\bm{a}_{t})/k$, where $\bm{a}_{t}=\big(\theta(\bm{a}):\theta\in\Theta_\Delta(t)\big)$.
For the proof of the existence, see~\cite[Theorem~5.4.1]{MLPK}.
\end{definition}  
  
\begin{lemma}\label{lem:Kolchin_poly_projection}
   There exists a computable function $\KolchinProj(N)$ such that for every
   \begin{itemize}
       \item differential variety $X \subset \mathbb{A}^n$ bounded by $N$,
       \item irreducible component $X_0 \subset X$,
       \item and linear projection $\pi\colon \mathbb{A}^n \to \mathbb{A}^\ell$,
   \end{itemize} 
   we have $|\omega_{Y}| \leqslant \KolchinProj(N)$, where $Y := \overline{\pi(X_0)}^{\operatorname{Kol}}$.
\end{lemma}
  
\begin{proof}
  By performing a linear change of variables, we reduce the problem to the case in which $\pi$ is the projection to the first $\ell$ coordinates.
  Consider a ranking such that 
  \begin{itemize}
      \item $y_{\ell + i}$ is greater than every derivative of $y_{j}$ for every $i > 0$ and $1 \leqslant j \leqslant \ell$;
      \item the restriction of the ranking on $y_1, \ldots, y_\ell$ is an orderly ranking (that is, a ranking such that $\ord\theta_1 > \ord\theta_2 $ implies, for all $i$ and $j$, $\theta_1 y_i > \theta_2 y_j$).
  \end{itemize}
  Theorem~\ref{thm:components} implies that $X_0$ has a characteristic set $\mathcal{C}$ with respect to this ranking with the order bounded by a computable function of $N$.
  Since a characteristic set of $Y$ can be obtained from $\mathcal{C}$ by selecting the polynomials only in the first $\ell$ variables, there is a charactersitic set of $Y$ with respect to the orderly ranking with the order bounded by a computable function of $N$.
  Then \cite[Proposition~3.1]{Omar} and~\cite[Fact~2.1]{Omar} imply that $|\omega_Y|$ is bounded by a computable function of $N$.
\end{proof}
  
\begin{proposition}\label{prop:chain_Kolchin}
  There exists an algorithm that, for every computable function $g(n)\colon \mathbb{Z}_{\geqslant 0} \to \mathbb{Z}_{\geqslant 0}$, produces a number $\Len_g$ such that, for every sequence of Kolchin polynomials
  \[
    \omega_0 > \omega_1 > \ldots > \omega_\ell
  \]
  such that $|\omega_i| < g(i)$ for every $0 \leqslant i \leqslant \ell$, we have $\ell < \Len_g$.
\end{proposition}

\begin{proof}
  By replacing $g(n)$ with $n + \max\limits_{0 \leqslant  s \leqslant n} g(s)$, we can further assume that $g(n)$ is increasing and $g(n) \geqslant n$.
  \cite[Definition~2.4.9 and Lemma~2.4.12]{MLPK} define a computable order-preserving map $c$ from the set of all Kolchin polynomials $\mathcal{K}$ to $\mathbb{Z}_{\geqslant 0}^{m + 1}$ (considered with respect to the lexicographic ordering).
  For $v = (v_0, \ldots, v_m) \in \mathbb{Z}_{\geqslant 0}^{m + 1}$, we define $|v| = v_0 + \ldots + v_m$.
  For every function $g\colon \mathbb{Z}_{\geqslant 0} \to \mathbb{Z}_{\geqslant 0}$, we define
  \[
  \widetilde{g}(n) := \max\limits_{\omega \in \mathcal{K},\; |\omega| \leqslant g(n)} |c(\omega)|.
  \]
  Note that if $g(n)$ was computable, then $\widetilde{g}(n)$ is also computable.
  
  The sequence $\omega_0 > \omega_1 > \ldots$ gives rise to a sequence $c(\omega_0) >_{\operatorname{lex}} c(\omega_1) >_{\operatorname{lex}} \ldots$ in $\mathbb{Z}_{\geqslant 0}^{m + 1}$ with $|c(\omega_i)| \leqslant \widetilde{g}(i)$ for every $i$.
  \cite[Main Lemma]{DixonBound} implies that there is an algorithm to compute the maximal length of such a sequence, so there is an algorithm to compute a bound on $\ell$ from $g$.
\end{proof}
\subsection{Trains of varieties, partial solutions, and their upper bounds}\label{sec:trains}

\begin{lemma}\label{lem:embedding}
 For every $\Delta$-$\sigma$-field $k$ of characteristic zero, there exists an extension $k \subset K$ of $\Delta$-$\sigma$-fields, where $K$ is a differentially closed $\Delta$-$\sigma^*$-field.
\end{lemma}
\begin{proof}
The proof follows \cite[Lemma~6.1]{ordinary_case} mutatis mutandis  as follows. We will show that there exists a $\Delta$-$\sigma^\ast$-field $K_0$ containing $k$.
  The proof of~\citep[Proposition~2.1.7]{Levin:difference} implies that one can build an ascending chain of $\sigma$-fields
  \begin{equation}\label{eq:asc_chain}
  k_0 \subset k_1 \subset k_2 \subset \ldots
  \end{equation}
  such that, for every $i \in \mathbb{N}$, there exists an isomorphism $\varphi_i\colon k \to k_i$ of $\sigma$-fields,  $\sigma(k_{i + 1}) = k_i$, and $\varphi_i = \sigma\circ\varphi_{i + 1}$ for every $i \in \mathbb{N}$.
  We transfer the $\Delta$-$\sigma$-structure from $k$ to $k_i$'s via $\varphi_i$'s.
  Then $\varphi_i = \sigma\circ \varphi_{i + 1}$ implies that the restriction of $\Delta$ on $k_{i + 1}$ to $k_i$ coincides with the action of $\Delta$ on $k_i$.
  We set $K_0 := \bigcup\limits_{i \in \mathbb{N}} k_i$.
  Since the action $\Delta$ and $\sigma$ is consistent with the ascending chain~\eqref{eq:asc_chain}, $K_0$ is a $\Delta$-$\sigma$-extension of $k_0 \cong k$.
  It is shown in~\cite[Proposition~2.1.7]{Levin:difference} that the action of $\sigma$ on $K_0$ is surjective. 
\cite[Corollary~2.4]{LS2016} implies that $K_0$ can be embedded in a differentially closed $\Delta$-$\sigma^\ast$-field $K$.
\end{proof}

\begin{notation}
 Within Sections~\ref{sec:trains} and~\ref{sec:delayPDEproof}, we fix a ground $\Delta$-$\sigma$ field $k$ and a differentially closed $\Delta$-$\sigma^\ast$-field $K$ given by Lemma~\ref{lem:embedding} applied to $k$.
 All varieties in Sections~\ref{sec:trains} and~\ref{sec:delayPDEproof} are considered over $K$.
\end{notation}

 \begin{definition}[Partial solutions] \label{def-partialsolution}
 \begin{itemize}
 \item[]
  \item
 For $\Delta$-$\sigma$-rings $\mathcal R_1$ and $\mathcal R_2$,
 a   homomorphism $\phi: \mathcal R_1\longrightarrow \mathcal R_2$ is called a $\Delta$-$\sigma$-homomorphism
 if, for all $i$, $\phi \partial_i=\partial_i\phi$ and $\phi \sigma=\sigma\phi$.
  \item Let $\mathcal R$ be a $\Delta$-$\sigma$-ring containing a $\Delta$-$\sigma$-field $k$. Let $k[\bm{y}_\infty]$ be the $\Delta$-$\sigma$-polynomial ring  over $k$ in $\bm{y}= y_1,\ldots,y_r$.
Given a point $\bm{a}=(a_1,\dots,a_r)\in \mathcal{R}^r$, there exists a unique $\Delta$-$\sigma$-homomorphism over $k$, 
\[\phi_{\bm{a}}: k[\bm{y}_\infty]\longrightarrow\mathcal R \quad \text{with}\ \ \phi_{\bm{a}}(y_i)=a_i\  \text{and}\ \phi_{\bm{a}}|_k=\id.\]
 Given $f\in k[\bm{y}_\infty]$, $\bm{a}$ is called  a solution of $f$ in $\mathcal R$ if $f\in\Ker(\phi_{\bm{a}})$.
\item  For a $\Delta$-$\sigma$-$k$-algebra $\mathcal{R}$ and $I= \mathbb{N}$ or  $\mathbb{Z}$, the sequence ring $\mathcal{R}^I$  has  the following structure of a $\Delta$-$\sigma$-ring ($\Delta$-$\sigma^\ast$-ring for $I=\mathbb{Z}$) with $\sigma$ and $\Delta$ defined by 
 \[\sigma\big((x_i)_{i\in I}\big):=(x_{i+1})_{i\in I}   \quad\text{and}\quad \partial_j\big((x_i)_{i\in I}\big):=(\partial_j(x_{i}))_{i\in I}.\] 
 For a $k$-$\Delta$-$\sigma$-algebra $\mathcal{R}$, $\mathcal{R}^I$ can be considered a $k$-$\Delta$-$\sigma$-algebra 
  by embedding $k$ into $\mathcal{R}^I$  in  the following way: 
  \[a\mapsto (\sigma^i(a))_{i\in I},\ \  a \in k.\]
For $f\in k[\bm{y}_\infty]$, a solution of $f$ with components in $\mathcal{R}^I$ is called a {\em sequence solution of $f$ in $\mathcal{R}$.} 
\item
 Given $f\in\mathcal{R}[\bm{y}_\infty]$, the order of $f$ is defined to be the maximal $\ord\theta+j$ such that $\theta\sigma^jy_s$ effectively appears in $f$ for some $s$, denoted by $\ord(f)$.
\item
 The relative order  of $f$ with respect to $\Delta$ (resp. $\sigma$), denoted by $\ord_\Delta(f)$ (resp. $\ord_\sigma(f)$),  is defined as the maximal $\ord\theta$ (resp. $j$) such that 
 $\theta\sigma^jy_s$ effectively appears in $f$ for some $s$.
 \item
   Let $F=\{f_1,\ldots,f_N\} \subset k[\bm{y}_\infty]$, where $\bm{y}= y_1,\ldots,y_r$, be a set of  $\Delta$-$\sigma$-polynomials.
   Suppose $h=\max\{\ord_\sigma(f) \mid f \in F\}$. 
   A sequence  of tuples $(\overline{a}_1,\ldots,\overline{a}_r)\in K^{\ell+h}\times\cdots\times K^{\ell+h}$ is called a {\em partial solution} of $F$ of length $\ell$ if
   $(\overline{a}_1,\ldots,\overline{a}_r)$ is a $\Delta$-solution of the system in $\bm{y}_{\infty,\ell+h-1}$:
   \[
   \{\sigma^i (F)=0 \;|\; 0\leqslant i\leqslant\ell-1\},
   \]
    where $\bm{y}_{\infty,\ell+h-1}=\{\theta\sigma^iy_s\mid \theta \in \Theta_\Delta;\,0\leqslant i \leqslant \ell+h-1;\,1\leqslant s\leqslant r\}$.
   \end{itemize}
  \end{definition}

  We associate the following geometric data with the above set $F$ of $\Delta$-$\sigma$-polynomials:
  \begin{itemize}
  \item the $\Delta$-variety $X\subset\mathbb A^H$ defined by $f_1=0,\ldots,f_N=0$ regarded as $\Delta$-equations in 
 $k[\bm{y}_{\infty,h}]$
  with $H={r}(h+1)$,
  and 
 \item two 
 projections $\pi_1,\pi_2:\mathbb A^H\longrightarrow \mathbb A^{H-r}$ defined by 
 \begin{align*}\pi_1(a_1,\ldots,\sigma^h(a_1);\ldots; &a_r,\ldots,\sigma^h(a_r))\\&:=(a_1,\sigma(a_1),\ldots,\sigma^{h-1}(a_1);\ldots; a_r,\ldots,\sigma^{h-1}(a_r)),\\
  \pi_2(a_1,\ldots,\sigma^h(a_1);\ldots; &a_r,\ldots,\sigma^h(a_r))\\&:=(\sigma(a_1),\ldots,\sigma^{h}(a_1);\ldots; \sigma(a_r),\ldots,\sigma^{h}(a_r)).
  \end{align*}
  \end{itemize}
  
  Let $\sigma(X)$ denote the $\Delta$-variety in $\mathbb A^H$ defined by $f_1^{\sigma},\ldots,f_N^{\sigma}$, 
  where $f_i^{\sigma}$ is the result by applying $\sigma$ to the coefficients of $f_i$.

 \begin{definition}
 A sequence $p_1,\ldots,p_\ell\in\mathbb A^H$ is a {\em partial solution of the triple} $(X,\pi_1,\pi_2)$  if 
\begin{enumerate}[label=(\arabic*)]
\item  for all $i$, $1\leqslant i \leqslant \ell$, we have $p_i\in\sigma^{i-1}(X)$ and
\item for all $i$, $1\leqslant i<\ell$, we have $\pi_1(p_{i+1})=\pi_2(p_i)$.
\end{enumerate}
 A two-sided infinite sequence with such a property is called a {\em solution of the triple} $(X,\pi_1,\pi_2)$.
 \end{definition}

 \begin{lemma}\label{lem:64}
 For every positive integer $\ell$, $F$ has a partial solution of length $\ell$ if and only if the triple $(X,\pi_1,\pi_2)$ has a partial solution of length $\ell.$
 The system $F$ has a  sequence solution in $K^\mathbb Z$ if and only if the triple $(X,\pi_1,\pi_2)$ has a   solution.
 \end{lemma}
 \begin{proof}
As in \cite[Lemma~6.5]{ordinary_case}.
 \end{proof}

  \begin{definition}
For $\ell\in\mathbb N$ or $+\infty$, a sequence of irreducible $\Delta$-subvarieties $(Y_1,\ldots,Y_\ell)$ in $\mathbb{A}^H$ is said to be {\em a train of length $\ell$} in $X$ if
\begin{enumerate}[label=(\arabic*)]
\item for all $i$, $1\leqslant i\leqslant \ell$, we have $Y_i\subseteq\sigma^{i-1}(X)$ and 
\item for all $i$, $1\leqslant i<\ell$, we have $\overline{\pi_1(Y_{i+1})}^{\kol}=\overline{\pi_2(Y_i)}^{\kol}$.
\end{enumerate}
 \end{definition}
 
  \begin{lemma}\label{lem:66}
 For every train $(Y_1,\ldots,Y_\ell)$ in $X$,  there exists a partial solution $p_1,\ldots,p_\ell$ of $(X,\pi_1,\pi_2)$ such that 
 for all $i$, we have $p_i\in Y_i$. In particular, if there is an infinite train in $X$, then there is a solution of the triple  $(X,\pi_1,\pi_2)$.
 \end{lemma} 
\begin{proof}
We prove it as in \cite[Lemma~6.7]{ordinary_case}, as follows. To prove the existence of a partial solution of $(X,\pi_1,\pi_2)$ with the desired property,
it suffices to prove the following: 

\begin{claim} \em There exists a nonempty open (in the sense of the Kolchin topology) subset $U\subseteq Y_\ell$  such that for each $p_\ell\in U$, $p_\ell$ 
can
be extended to a partial solution $p_1,\ldots,p_\ell$ of  
 $(X,\pi_1,\pi_2)$ with $p_i\in Y_i$  for every $1 \leqslant i \leqslant \ell$.
 \end{claim}
 
 We will prove the Claim by induction on $\ell$.
 For $\ell=1$, take $U=Y_1$. 
 Since each point in $Y_1$ is a partial solution of $(X,\pi_1,\pi_2)$ of length 1, the Claim holds for $\ell=1$.
 Now suppose we have proved the Claim for $\ell-1$.  
 So there exists a nonempty open subset $U_0\subseteq Y_{\ell-1}$    satisfying the desired property.
 Since $Y_{\ell-1}$ is irreducible, $U_0$ is dense in $Y_{\ell-1}.$
 So, $\pi_2(U_0)$ is dense in $\overline{\pi_2(Y_{\ell-1})}^{\kol}=\overline{\pi_1(Y_{\ell})}^{\kol}$. 
Since $U_0$ is $\Delta$-constructible (that is, solution set of a quantifier-free formula with parameters in $K$  or, equivalently, a finite union of $\Delta$-closed and $\Delta$-open sets), $\pi_2(U_0)$ is $\Delta$-constructible too.
So, $\pi_2(U_0)$ contains a nonempty open subset of $\overline{\pi_1(Y_{\ell})}^{\kol}$.

Since $\pi_1(Y_{\ell})$ is $\Delta$-constructible  and dense in $\overline{\pi_1(Y_{\ell})}^{\kol}$, 
$\pi_2(U_0)\cap\pi_1(Y_{\ell})\neq\varnothing$ is $\Delta$-constructible  and dense in $\overline{\pi_1(Y_{\ell})}^{\kol}$. 
Let $U_1$ be a nonempty open subset of $\overline{\pi_1(Y_{\ell})}^{\kol}$ contained in $\pi_2(U_0)\cap\pi_1(Y_{\ell})$ and \[U_2=\pi_1^{-1}(U_1)\cap Y_{\ell}.\]
Then $U_2$ is a nonempty open subset of  $Y_{\ell}$. 
We will show that for each $p_\ell\in U_2$, there exists $p_{i}\in Y_i$ for $i=1,\ldots,\ell-1$ such that $p_1,\ldots,p_{\ell}$ is a partial solution of  $(X,\pi_1,\pi_2).$ 
 
Since $\pi_1(p_\ell)\in U_1\subset\pi_2(U_0)$, there exists $p_{\ell-1}\in U_0$ such that $\pi_1(p_\ell)=\pi_2(p_{\ell-1}).$ 
Since  $p_{\ell-1}\in U_0$, by the inductive hypothesis, there exists $p_{i}\in Y_i$ for $i=1,\ldots,\ell-1$ such that $p_1,\ldots,p_{\ell-1}$ is a partial solution of  $(X,\pi_1,\pi_2)$ of length $\ell-1$.
So $p_1,\ldots,p_{\ell}$ is a partial solution of  $(X,\pi_1,\pi_2)$ of length $\ell$.
  \end{proof}

For two trains $Y=(Y_1,\ldots,Y_\ell)$ and $Y'=(Y'_1,\ldots,Y'_\ell)$, denote $Y\subseteq Y'$ if $Y_i\subseteq Y'_i$ for each $i$.
 Given an increasing chain of trains $Y_{i}=(Y_{i,1},\ldots,Y_{i,\ell})$,
 \[\big(\overline{\cup_iY_{i,1}}^{\kol},\ldots,\overline{\cup_iY_{i,\ell}}^{\kol}\big)\] is a train in $X$ that is an upper bound for this chain. 
 (For each $j$, $\overline{\cup_iY_{i,j}}^{\kol}$ is an irreducible $\Delta$-variety in $\sigma^{j-1}$(X).)
 So by Zorn's lemma, maximal trains of length $\ell$ always exist in $X$.
 
For $\ell\in\mathbb N$, consider the product 
\[
\textbf{X}_\ell:=X\times\sigma(X)\times\ldots\times \sigma^{\ell-1}(X)
\] 
and denote the projection of $\textbf{X}_\ell$ onto $\sigma^{i-1}(X)$ by $\varphi_{\ell,i}$. 
 Let 
 \begin{equation*}
     \label{eq:W}
\mathbf{W}_\ell(X,\pi_1,\pi_2):=\{p\in \textbf{X}_\ell: \pi_2(\varphi_{\ell,i}(p))=\pi_1(\varphi_{\ell,i+1}(p)), i=1,\ldots,\ell-1\}.
 \end{equation*}
 
 \begin{lemma} \label{lm-traincorrespondence}
 For every irreducible $\Delta$-subvariety $W\subset \mathbf{W}_\ell$, \[\big(\overline{\varphi_{\ell,1}(W)}^{\kol},\ldots,\overline{\varphi_{\ell,\ell}(W)}^{\kol}\big)\] is a train in $X$ of length $\ell$.
 Conversely, for each train $(Y_1,\ldots,Y_\ell)$ in $X$, 
 there exists an irreducible $\Delta$-subvariety $W\subseteq \mathbf{W}_\ell$ such that $Y_i=\overline{\varphi_{\ell,i}(W)}^{\kol}$ for each $i=1,\ldots,\ell$.
 \end{lemma}
\begin{proof} The proof follows  \cite[Lemma~6.8]{ordinary_case}. The first assertion is straightforward. 
  We will prove the second assertion  by induction on $\ell$.
  For $\ell = 1$, $\mathbf{W}_1= X$, and we can set $W = Y_1$.
  
  Let $\ell > 1$. Apply the inductive hypothesis to the train $(Y_1, \ldots, Y_{\ell - 1})$  
  and obtain an irreducible subvariety $Y' \subset \mathbf{W}_{\ell-1} \subset \mathbf{X}_{\ell - 1}$.
  Then there is a natural embedding of $Y' \times Y_\ell$ into $\mathbf{X}_\ell$.
  Denote $(Y' \times Y_\ell) \cap \mathbf{W}_\ell$ by $\widetilde Y$.
  Since $Y'\subset \mathbf{W}_{\ell-1}$,
  \begin{equation*}\label{eq:reprW}
 \widetilde Y = \{ p \in Y' \times Y_\ell \:|\: \pi_2\left( \varphi_{\ell, \ell - 1}(p) \right) = \pi_1\left( \varphi_{\ell, \ell}(p) \right) \}.
 \end{equation*}
Let 
  \begin{equation}\label{eq:reprZ}
  Z := \overline{\pi_2\left( \varphi_{\ell - 1, \ell - 1}(Y') \right)} = \overline{\pi_1(Y_\ell)}.
  \end{equation}
  Then we have a $(k, \Delta)$-isomorphism 
  \[
  R_{Y'}\otimes_{R_Z}R_{Y_\ell} \to R_{\widetilde{Y}}
  \]  
  under the $(k,\Delta)$-algebra homomorphisms   
  $i_1:R_Z \to R_{Y'}$ and $i_2:R_Z \to R_{Y_\ell}$ induced by  $\pi_2\circ \varphi_{\ell - 1, \ell - 1} $ and $\pi_1$, respectively.
  Equality~\eqref{eq:reprZ} implies that $i_1$ and $i_2$ are injective.
 Denote the fields of fractions of $R_{Y'}$, $R_{Y_\ell}$, and $R_Z$ by $E$, $F$, and $L$, respectively.
  Let $\mathfrak{p}$ be any prime 
  differential ideal in $E \otimes_L F$, \[R := (E \otimes_L F) / \mathfrak{p},\]and $\pi\colon E \otimes_L F \to R$ be the canonical homomorphism.
  Consider the natural homomorphism $i \colon R_{Y'}\otimes_{R_Z}R_{Y_\ell} \to E \otimes_L F$.
  Since $1 \in i(R_{Y'}\otimes_{R_Z}R_{Y_\ell})$, the composition $\pi \circ i$ is a nonzero homomorphism.
Since $i_1$ and $i_2$ are injective, the natural homomorphisms $i_{Y'} \colon R_{Y'} \to R_{Y'}\otimes_{R_Z}R_{Y_\ell}$ and $i_{Y_\ell} \colon R_{Y_\ell} \to R_{Y'}\otimes_{R_Z}R_{Y_\ell}$ are injective as well.
  We will show that the compositions \[\pi\circ i \circ i_{Y'} \colon R_{Y'} \to R\quad \text{and}\quad 
  \pi\circ i \circ i_{Y_\ell} \colon R_{Y_\ell} \to R\] are injective.
  Introducing the natural embeddings $i_E \colon E \to E\otimes_L F$ and $j_{Y'} \colon R_{Y'} \to E$, we can rewrite
  \[
  \pi\circ i \circ i_{Y'} = \pi \circ i_E \circ j_{Y'}.
  \]
  The homomorphisms $i_E$ and $j_{Y'}$ are injective.
  The restriction of $\pi$ to $i_E(E)$ is also injective since $E$ is a field.
  Hence, the whole composition $\pi \circ i_E \circ j_{Y'}$ is injective. 
  The argument for $\pi\circ i \circ i_{Y_\ell}$ is analogous. Let 
  \[S :=\big(R_{Y'}\otimes_{R_Z}R_{Y_\ell}\big)\big/ \bigl(\mathfrak{p}\cap \big(R_{Y'}\otimes_{R_Z}R_{Y_\ell}\big)\bigr),\]
  which is a domain, and the  homomorphisms $\pi\circ i \circ i_{Y'}: R_{Y'}\to S$ and $\pi\circ i \circ i_{Y_\ell} : R_{Y_\ell}\to S$ are injective.  Let $F \subset k\{\textbf{W}_\ell\}$ be such that $S = k\{\textbf{W}_\ell\}/\sqrt{\langle F\rangle}^{(\infty)}$. We now let $W$  be the $\Delta$-subvariety of $\textbf{W}_\ell$ defined by $F =0 $.
  For every $i$, $1 \leqslant i < \ell$, the homomorphism \[\varphi_{\ell, i}^\sharp = (\pi\circ i \circ i_{Y'})\circ\varphi_{\ell - 1, i}^\sharp : R_{Y_i}\to R_{Y'}\to S \] is injective as a composition of two injective homomorphisms. Hence,
  the restriction $\varphi_{\ell, i} \colon W \to Y_i$ is dominant.
 \end{proof}

\begin{lemma}\label{lem:numtrainsnumcomponents}
  Let $(X, \pi_1, \pi_2)$ be a triple with $X$ bounded by $n$.
  Then, for every $\ell$, the number of maximal trains of length $\ell$ in $X$ does not exceed $\Components(m,\ell n)$.
\end{lemma}

\begin{proof}
 By Lemma~\ref{lm-traincorrespondence}, the number of maximal trains of length $\ell$ in $X$ is equal to the number of irreducible components of $\textbf{W}_\ell$. By  Theorem~\ref{thm:components}, this number does not exceed $\Components(m,\ell n)$.
\end{proof}

\begin{definition}
  Let $(X, \pi_1, \pi_2)$ be a triple and $\omega(t)$ be a numeric polynomial.
  We define $B(X, \omega) \in \mathbb{Z} \cup \{\infty\}$ as the smallest value that is greater than the length of any train in $X$ with Kolchin polynomials at least~$\omega$.
\end{definition}

\begin{lemma}\label{lem:train_components}
  Let $X$ be a differential variety bounded by $n$ such that $B(X, 0) < \infty$.
  Then $B(X, \omega_X)$ does not exceed the number of components of $X$ plus one.
\end{lemma}

\begin{proof}
  Denote the number of components in $X$ by $N$ and assume that there is a train $(Y_1, \ldots, Y_{N + 1})$ with the Kolchin polynomial at least $\omega_X$.
  Then each of $Y_1, \sigma^{-1}(Y_2), \ldots, \sigma^{-N}(Y_{N + 1})$ must be a component of $X$, so there exist $1 \leqslant i < j \leqslant N + 1$ such that  $Y_j = \sigma^{j-i}Y_i$.
  Thus, there exists an infinite train
 $(Y_1,\ldots,Y_i,Y_{i+1},\ldots, Y_{j-1}, \sigma^{j - i }(Y_i), \sigma^{j - i}(Y_{i + 1}), \ldots)$ in $X$.
  This contradicts to  $B(X, 0) < \infty$.
\end{proof}

\begin{lemma}\label{lem:next_Kolchin}
  There exists a computable function $\Iter(n, D)$ such that, for every triple $(X, \pi_1, \pi_2)$ such that
  \begin{itemize}
      \item $B(X, 0) < \infty$
      \item $X$ is bounded by $n$
  \end{itemize}
   and every numeric polynomial $\omega_1(t) > 0$, there exists a numeric polynomial $\omega_2(t) \geqslant 0$ such that
  \begin{itemize}
      \item $\omega_2(t) < \omega_1(t)$;
      \item $|\omega_2| \leqslant \Iter(n, B(X, \omega_1))$;
      \item $B(X, \omega_2) \leqslant \Iter(n, B(X, \omega_1))$.
  \end{itemize}
\end{lemma}

\begin{proof}
  The proof follows~\cite[Lemma~6.20]{ordinary_case}.
  Let $B_1 := B(X, \omega_1)$, and let $T$ be the number of maximal trains of length $B_1$ in  $X$.
  We set $B_2 := B_1 + T$.
  Lemma~\ref{lem:numtrainsnumcomponents} implies that $T$ is bounded by $\Components(m,nB_1)$.
  Consider the fibered product $\mathbf{W}_{B_1}(X, \pi_1, \pi_2)$, and, for each irreducible component $W$ in it, denote the corresponding train by $Y_W$.
  We set (assuming $\max\varnothing = 0$)
  \[
    \omega_2 := \max \bigl\{ \omega_{Y_W} \mid \omega_{Y_W} < \omega_1, W\text{ is a component of }\mathbf{W}_{B_1} \bigr\}.
  \]
  We will show that $B(X, \omega_2) \leqslant B_1 + T$.
  Assume that there is a maximal train $(Y_1, \ldots, Y_{B_2})$ in $X$ with the Kolchin polynomial at least $\omega_2$.
  Introduce $T + 1$ trains $Z^{(1)},\ldots,Z^{(T + 1)}$ of length $B_1$ in $X,\sigma(X),\ldots,\sigma^{T}(X)$, respectively, such that for each $j$,
  \[
    Z^{(j)}=\big( Z^{(j)}_1,\ldots,Z^{(j)}_{T}\big) := (Y_j, \ldots, Y_{j + B_1 - 1}).
  \]
  Then for each $j$, consider a maximal train $\tilde{Z}^{(j)}$ of length $B_1$ containing $Z^{(j)}$.
  So $\sigma^{-j+1}(\tilde{Z}^{(j)})$ is a maximal train of length $B_1$ in $X$.
  There are two cases to consider:
  \begin{equation}
  \tag{Case 1}
  \big\{\omega_{Y_W}(t)\:\big|\:  \omega_{Y_W}(t)<\omega_1(t),\;  W \text{  is a component of } 
 \mathbf{W}_{B_1}\big\}=\varnothing.
  \end{equation}
  In this case, 
  $\tilde{Z}^{(1)}$ is a train in $X$ with  Kolchin polynomial at least $\omega_1$.
  This contradicts the definition of $B(X, \omega_1)$.
  
  \begin{equation}
  \tag{Case 2}
  \big\{\omega_{Y_W}(t)\:\big|\:\omega_{Y_W}(t)<\omega_1(t),\;  W \text{  is a component of }
\mathbf{W}_{B_1}
\big\}\neq\varnothing.
  \end{equation}
  By the definition of $B(X, \omega_1)$, for every $j$, 
  $\omega_{\sigma^{-j+1}(\tilde{Z}^{(j)})}(t) < \omega_1(t)$.
  This implies that, for each $j$, 
  \[
    \omega_{\sigma^{-j+1}(\tilde{Z}^{(j)})}(t)=\omega_{2}(t).
  \]
  Since there are only $T$ maximal trains in  $X$ of length $B_1$,
  there exist $a < b$ such that 
  \[
    \sigma^{-a+1}(\tilde{Z}^{(a)})=\sigma^{-b+1}(\tilde{Z}^{(b)}) =: Z.
  \]
  Since $\omega_{Z} = \omega_{2}$, there exists $\ell$ such that $\omega_{Z_\ell} = \omega_2$.
  Since
  \[
    \omega_{\sigma^{-a+1}({Z}^{(a)}_\ell)} = \omega_{2} \quad \text{ and }\quad \sigma^{-a+1}({Z}^{(a)}_\ell) \subseteq Z_\ell
  \] 
  we have 
$
    \sigma^{-a+1}({Z}^{(a)}_\ell) = Z_\ell
$.
  Similarly, we can show 
 $  \sigma^{-b+1}({Z}^{(b)}_\ell) = Z_\ell$.
 Hence,
 \[
   \sigma^{-a+1}(Y_{a + \ell - 1}) = \sigma^{-a+1}({Z}^{(a)}_\ell) = \sigma^{-b+1}({Z}^{(b)}_\ell) = \sigma^{-b+1}(Y_{b+\ell - 1}).
 \]
  Thus, we have 
 $
    Y_{b+\ell-1}=\sigma^{b-a}(Y_{a+\ell-1})
$.
  This contradicts the fact that $B(X, 0) < \infty$.
  
  It remains to show that $|\omega_2|$ is bounded by a computable function of $n$ and $B_1$.
  Let $W$ be a component of $\mathbf{W}_{B_1}$ such that $\omega_{Y_W} = \omega_2$.
  Let $Y_W = (Y_{W, 1}, \ldots, Y_{W, B_1})$.
  There exists $1 \leqslant i \leqslant B_1$ such that $\omega_{Y_i} = \omega_2$.
  Since $Y_i$ is the Kolchin closure of a linear projection of a component of $\mathbf{W}_{B_1}$ and $\mathbf{W}_{B_1}$ is bounded by $B_1n$, Lemma~\ref{lem:Kolchin_poly_projection} implies that $|\omega_2|$ is bounded by a computable function of $n$ and $B_1$.
  
  Taking $\Iter(n, D)$ to be the maximum of the computable bounds for $B(X, \omega_2)$ and $|\omega_2|$, we conclude the proof.
\end{proof}


\begin{definition}
  Let $n$ be a positive integer and $\omega(t)$ be a numeric polynomial such that $\omega > 0$.
  We define $B(n, \omega) \in \mathbb{Z} \cup \{\infty\}$ as the smallest value such that, for every affine differential variety $X$ bounded by $n$, if there exists a train in $X$ with Kolchin polynomial at least $\omega$ of length at least $B(n, \omega)$, then there exists an infinite train in $X$.
\end{definition}

\begin{proposition}\label{prop:BboundA}
  $B(n, 0)$ is bounded by a computable function $A(n)$.
\end{proposition}

\begin{proof}
  We recursively define the following function $G(n)$ on nonnegative integers
 \begin{align*}
    &G(0) := \max\bigl(\Components(n,n) + 1, \KolchinProj(n) \bigr),\\
    &G(j + 1) := \Iter(n, G(j)),\ \ j\geqslant 0.
 \end{align*}
  Consider a variety $X$ bounded by $n$ such that there is no infinite train in $X$, that is $B(X, 0) < \infty$.
  Lemma~\ref{lem:train_components} implies that $B(X, \omega_X) - 1$ does not exceed the number of components of $X$.
  Hence, Theorem~\ref{thm:components} implies that $B(X, \omega_X) \leqslant \Components(n,n) + 1$. 
  Lemma~\ref{lem:Kolchin_poly_projection} implies that $|\omega_X| \leqslant \KolchinProj(n)$.
  Repeatedly applying Lemma~\ref{lem:next_Kolchin}, we obtain a sequence of numeric polynomials
  \[
    \omega_0 := \omega_X > \omega_1 > \omega_2 > \ldots 
  \]
  such that, for every $1 \leqslant i \leqslant L$, we have $B(X, \omega_i) \leqslant G(i)$ and $|\omega_i| \leqslant G(i)$.
  Since the Kolchin polynomial are well-ordered, there exists $L$ such that $\omega_L = 0$.
  Proposition~\ref{prop:chain_Kolchin} implies that $L \leqslant \Len_G$.
  Hence, $B(X, 0) \leqslant G(\Len_G)$,  where the right-hand side is a computable function of $n$. 
 Set $A(n):=G(\Len_G)$, then $B(n,0)\leqslant A(n)$.
\end{proof}

 \begin{corollary} \label{cor:620}
For all $r$, $m$ and $s\in\mathbb{Z}_{\geqslant 0}$, and a set of $\Delta$-$\sigma$ polynomials $F \subset k[\bm{y}_{s}]$ with $|\Delta|=m$, $\deg F\leqslant s$ and $|\bm{y}|= r$,
 $F = 0$ has a  sequence solution in 
 $K^\mathbb Z$ 
 if and only if $F = 0$ has a partial solution of computable length
 $A(\max\{r,m,s\})$.
 \end{corollary}
\begin{proof}
 The proof is as in \cite[Corollary~6.21]{ordinary_case}, as follows.  Let $X\subset\mathbb A^H$ be the $\Delta$-variety defined by $F=0$ regarded as a system of $\Delta$-equations in $\bm{y}, \sigma(\bm{y}),\ldots, \sigma^h(\bm{y})$, where $H=n(h+1)$.
 By Lemmas~\ref{lem:64}  and~\ref{lem:66}, 
 $F = 0$ has a partial solution of length $D$ (resp. $F = 0$ has a solution in $K^\mathbb Z$
 ) if and only if there exists a train of length $D$ in $X$ (resp., there exists an infinite train in $X$).
 By Proposition~\ref{prop:BboundA}, if there exists a train of length $D := A(\max\{r,m,s\})$ in $X$, then there  exists a infinite train in $X$. So the assertion holds.
\end{proof}


\subsection{Upper bound for delay PDEs}\label{sec:delayPDEproof}

 We now state and prove 
the main result of this section
which generalizes \cite[Theorem~3.1]{ordinary_case} to delay PDEs. 
\begin{theorem}[Effective elimination for delay PDEs]\label{thm:main3}  For all non-negative integers $r$,  $m$, and $s$, there exists a computable $B=B(r,m,s)$ such that, for all:  \begin{itemize}
\item non-negative integers $q$ and $t$,
   \item $\Delta$-$\sigma$-fields
    $k$ with $\Char k =0$ and $|\Delta|=m$,   
\item sets of $\Delta$-$\sigma$-polynomials $F\subset k[\bm{x}_{t},\bm{y}_{s}]$, where $\bm{x} =x_1,\ldots,x_q$,  $\bm{y}=y_1,\ldots,y_r$, and
$\deg_{\bm{y}}F\leqslant s$, 
   \end{itemize}
   we have 
      \[\big\langle \sigma^i(F)\mid i\in \mathbb{Z}_{\geqslant 0} \big\rangle^{(\infty)}\cap k[\bm{x}_\infty]\ne\{0\} \iff \langle \sigma^i(F)\mid i\in [0,B] \big\rangle^{(B)}\cap k[\bm{x}_{B+t}]\ne\{0\}.\]
 \end{theorem}
 \begin{proof} The proof closely follows \cite[Theorem~6.22]{ordinary_case}.
 The ``$\impliedby$'' implication is straightforward. We will prove the ``$\implies$" implication.  For this, let 
 $A :=A(\max\{r, m, s\})$ from Corollary~\ref{cor:620}, and let
 $B$ be a computable bound obtained from  \cite[Theorem~3.4]{Gustavson} with \[m\leftarrow m,
 \ \  
 n\leftarrow r(A + s + 1),\ \ h\leftarrow s,\ \ \text{and}\ \ D \leftarrow s.\]
 By assumption, 
 \begin{equation}\label{eq:assumption}1 \in \big\langle \sigma^i(F)\mid i\in \mathbb{Z}_{\geqslant 0} \big\rangle^{(\infty)}\cdot k(\bm{x}_\infty)[\bm{y}_\infty].
 \end{equation}
  Suppose that 
 \begin{equation}\label{eq:zerointersect}\langle \sigma^i(F)\mid i\in [0, A] \big\rangle^{(B)}\cap k[\bm{x}_{B+t}]=\{0\}.
 \end{equation}
 If  \[1 \in \big\langle \sigma^i(F)\mid i\in [0, A] \big\rangle^{({B})}\cdot k(\bm{x}_{B+t})[\bm{y}_{\infty,A + s}],\] 
 then there would exist $c_{i, j}\in k(\bm{x}_{B+t})[\bm{y}_{\infty,A + s}]$ such that 
 \begin{equation}\label{eq:expr1}
  1 = \sum\limits_{\theta\in\Theta_{\Delta}(B)} \sum\limits_{j = 0}^A \sum\limits_{f\in F} c_{i, j}\theta(\sigma^j(f)).
 \end{equation}
 Multiplying equation~\eqref{eq:expr1} by the common denominator in the variables $\bm{x}_{B+t}$, we obtain a contradiction with~\eqref{eq:zerointersect}.  Hence,
 by \cite[Theorem~3.4]{Gustavson},
 \[1 \notin \big\langle \sigma^i(F)\mid i\in [0, A] \big\rangle^{(\infty)}\cdot k(\bm{x}_{B+t})[\bm{y}_{\infty, A + s}].\] 
  By 
  Lemma~\ref{lem:embedding}, there exists a differentially closed $\Delta$-$\sigma^*$-field extension $L \supset k(\bm{x}_{\infty})\supset k(\bm{x}_{B+t})$. 
 Then differential Nullstellensatz implies that the system of differential equations
 \[
 \{\sigma^i(F)=0\mid i\in [0, A]\}
 \] 
 in the unknowns $\bm{y}_{\infty, A + s}$ has a 
 solution in $L$.  
Then the system $F = 0$ has a partial solution of length $A + 1$ in $L$.  
Now from~\eqref{eq:assumption}, we see that the system $F=0$ has no solutions in $L^{\mathbb{Z}}$.
Together with the existence of a partial solution of length $A + 1$, this  contradicts to Corollary~\ref{cor:620}.
 \end{proof}

\section*{Acknowledgments}

This work was partially supported by the NSF grants  CCF-1564132, CCF-1563942, DMS-1760448, DMS-1760413, DMS-1853650, DMS-1853482, and DMS-1800492;  the NSFC grants (11971029, 11688101) and the fund of Youth Innovation Promotion Association of CAS.
 We are grateful to the referees for numerous insightful comments, which helped us  substantially improve the paper.

\bibliographystyle{abbrvnat}
\bibliography{bibdata}

\begin{thebibliography}{27}
\providecommand{\natexlab}[1]{#1}
\providecommand{\url}[1]{\texttt{#1}}
\expandafter\ifx\csname urlstyle\endcsname\relax
  \providecommand{\doi}[1]{doi: #1}\else
  \providecommand{\doi}{doi: \begingroup \urlstyle{rm}\Url}\fi

\bibitem[B{\"a}chler et~al.(2012)B{\"a}chler, Gerdt, Lange-Hegermann, and
  Robertz]{thomas}
T.~B{\"a}chler, V.~Gerdt, M.~Lange-Hegermann, and D.~Robertz.
\newblock Algorithmic {T}homas decomposition of algebraic and differential
  systems.
\newblock \emph{Journal of Symbolic Computation}, 47\penalty0 (10):\penalty0
  1233--1266, 2012.
\newblock URL \url{https://doi.org/10.1016/j.jsc.2011.12.043}.

\bibitem[Blum et~al.(1998)Blum, Cucker, Shub, and Smale]{BCSS1998}
L.~Blum, F.~Cucker, M.~Shub, and S.~Smale.
\newblock \emph{Complexity and Real Computation}.
\newblock Springer, 1998.
\newblock URL \url{https://doi.org/10.1007/978-1-4612-0701-6}.

\bibitem[Boulier et~al.(2009)Boulier, Lazard, Ollivier, and Petitot]{BLOP09}
F.~Boulier, D.~Lazard, F.~Ollivier, and M.~Petitot.
\newblock Computing representations for radicals of finitely generated
  differential ideals.
\newblock \emph{Applicable Algebra in Engineering, Communication and
  Computing}, 20:\penalty0 73--121, 2009.
\newblock URL \url{https://doi.org/10.1007/s00200-009-0091-7}.

\bibitem[Chapuis and Koiran(1999)]{CK99}
O.~Chapuis and P.~Koiran.
\newblock Saturation and stability in the theory of computation over the reals.
\newblock \emph{Annals of Pure and Applied Logic}, 99\penalty0 (1):\penalty0
  1--49, 1999.
\newblock URL \url{https://doi.org/10.1016/S0168-0072(98)00060-8}.

\bibitem[Gao et~al.(2009{\natexlab{a}})Gao, Luo, and Yuan]{GLY09}
X.-S. Gao, Y.~Luo, and C.~Yuan.
\newblock A characteristic set method for ordinary difference polynomial
  systems.
\newblock \emph{Journal of Symbolic Computation}, 44\penalty0 (3):\penalty0
  242--260, 2009{\natexlab{a}}.
\newblock URL \url{https://doi.org/10.1016/j.jsc.2007.05.005}.

\bibitem[Gao et~al.(2009{\natexlab{b}})Gao, Van~der Hoeven, Yuan, and
  Zhang]{GHYZ2009}
X.-S. Gao, J.~Van~der Hoeven, C.-M. Yuan, and G.-L. Zhang.
\newblock Characteristic set method for differential-difference polynomial
  systems.
\newblock \emph{Journal of Symbolic Computation}, 44\penalty0 (9):\penalty0
  1137--1163, 2009{\natexlab{b}}.
\newblock URL \url{https://doi.org/10.1016/j.jsc.2008.02.010}.

\bibitem[Gerdt and Robertz(2019)]{GR2019}
V.~Gerdt and D.~Robertz.
\newblock Algorithmic approach to strong consistency analysis of finite
  difference approximations to {PDE} systems.
\newblock In \emph{ISSAC '19: Proceedings of the 2019 on International
  Symposium on Symbolic and Algebraic Computation}, pages 163--170. ACM Press,
  2019.
\newblock URL \url{https://doi.org/10.1145/3326229.3326255}.

\bibitem[Golubitsky et~al.(2009)Golubitsky, Kondratieva, Ovchinnikov, and
  Szanto]{GKOS09}
O.~Golubitsky, M.~Kondratieva, A.~Ovchinnikov, and A.~Szanto.
\newblock A bound for orders in differential {N}ullstellensatz.
\newblock \emph{Journal of Algebra}, 322\penalty0 (11):\penalty0 3852--3877,
  2009.
\newblock URL \url{https://doi.org/10.1016/j.jalgebra.2009.05.032}.

\bibitem[Goode(1994)]{Poizat}
J.~B. Goode.
\newblock Accessible telephone directories.
\newblock \emph{The Journal of Symbolic Logic}, 59:\penalty0 92--105, 1994.
\newblock URL \url{https://www.jstor.org/stable/2275252}.

\bibitem[Gustavson et~al.(2016)Gustavson, Kondratieva, and
  Ovchinnikov]{Gustavson}
R.~Gustavson, M.~Kondratieva, and A.~Ovchinnikov.
\newblock New effective differential {N}ullstellensatz.
\newblock \emph{Advances in Mathematics}, 290:\penalty0 1138--1158, 2016.
\newblock URL \url{http://dx.doi.org/10.1016/j.aim.2015.12.021}.

\bibitem[Harrison-Trainor et~al.(2012)Harrison-Trainor, Klys, and
  Moosa]{HTKM12}
M.~Harrison-Trainor, J.~Klys, and R.~Moosa.
\newblock Nonstandard methods for bounds in differential polynomial rings.
\newblock \emph{Journal of Algebra}, 360:\penalty0 71--86, 2012.
\newblock URL \url{https://doi.org/10.1016/j.jalgebra.2012.03.013}.

\bibitem[Heintz(1983)]{Heintz}
J.~Heintz.
\newblock Definability and fast quantifier elimination in algebraically closed
  fields.
\newblock \emph{Theoretical Computer Science}, 24\penalty0 (3):\penalty0
  239--277, 1983.
\newblock URL \url{http://dx.doi.org/10.1016/0304-3975(83)90002-6}.

\bibitem[Kolchin(1973)]{Kol}
E.~Kolchin.
\newblock \emph{Differential Algebra and Algebraic Groups}.
\newblock Academic Press, New York, 1973.

\bibitem[L\'eon~S\'anchez(2016)]{LS2016}
O.~L\'eon~S\'anchez.
\newblock On the model companion of partial differential fields with an
  automorphism.
\newblock \emph{Israel Journal of Mathematics}, 212:\penalty0 419--442, 2016.
\newblock URL \url{https://doi.org/10.1007/s11856-016-1292-y}.

\bibitem[Le{\'o}n~S{\'a}nchez(2019)]{Omar}
O.~Le{\'o}n~S{\'a}nchez.
\newblock Estimates for the coefficients of differential dimension polynomials.
\newblock \emph{Mathematics of Computation}, 88:\penalty0 2959--2985, 2019.
\newblock URL \url{https://doi.org/10.1090/mcom/3429}.

\bibitem[L\'eon~S\'anchez and Moosa(2016)]{MLS2016}
O.~L\'eon~S\'anchez and R.~Moosa.
\newblock The model companion of differential fields with free operators.
\newblock \emph{Journal of Symbolic Logic}, 81\penalty0 (2):\penalty0 493--509,
  2016.
\newblock URL \url{https://doi.org/10.1017/jsl.2015.76}.

\bibitem[Levin(2008)]{Levin:difference}
A.~Levin.
\newblock \emph{Difference algebra}, volume~8 of \emph{Algebra and
  Applications}.
\newblock Springer, New York, 2008.
\newblock URL \url{http://dx.doi.org/10.1007/978-1-4020-6947-5}.

\bibitem[Li et~al.(2018)Li, Ovchinnikov, Pogudin, and Scanlon]{ordinary_case}
W.~Li, A.~Ovchinnikov, G.~Pogudin, and T.~Scanlon.
\newblock Elimination of unknowns for systems of algebraic
  differential-difference equations, 2018.
\newblock URL \url{https://arxiv.org/abs/1812.11390}.

\bibitem[Marker(2002)]{Marker}
D.~Marker.
\newblock \emph{Model theory: An introduction}.
\newblock Springer-Verlag New York, 2002.
\newblock URL \url{https://doi.org/10.1007/b98860}.

\bibitem[McAloon(1984)]{DixonBound}
K.~McAloon.
\newblock Petri nets and large finite sets.
\newblock \emph{Theoretical Computer Science}, 32\penalty0 (1-2):\penalty0
  173--183, 1984.
\newblock URL \url{https://doi.org/10.1016/0304-3975(84)90029-X}.

\bibitem[McGrail(2000)]{mcgrail}
T.~McGrail.
\newblock The model theory of differential fields with finitely many commuting
  derivations.
\newblock \emph{Journal of Symbolic Logic}, 65\penalty0 (2):\penalty0 885--913,
  2000.
\newblock URL \url{https://doi.org/10.2307/2586576}.

\bibitem[Mikhalev et~al.(1999)Mikhalev, Levin, Pankratiev, and
  Kondratieva]{MLPK}
A.~V. Mikhalev, A.~Levin, E.~Pankratiev, and M.~Kondratieva.
\newblock \emph{Differential and Difference Dimension Polynomials}, volume 461
  of \emph{Mathematics and Its Applications}.
\newblock Springer Netherlands, 1999.
\newblock URL \url{https://doi.org/10.1007/978-94-017-1257-6}.

\bibitem[Moosa and Scanlon(2014)]{MS2014}
R.~Moosa and T.~Scanlon.
\newblock Model theory of fields with free operators in characteristic zero.
\newblock \emph{Journal of Mathematical Logic}, 14\penalty0 (2):\penalty0
  1450009, 2014.
\newblock URL \url{https://doi.org/10.1142/s0219061314500093}.

\bibitem[Papadimitriou(1993)]{Papadimitriou}
C.~H. Papadimitriou.
\newblock \emph{Computational Complexity}.
\newblock Pearson, 1993.

\bibitem[Rust and Reid(1997)]{rankings}
C.~J. Rust and G.~J. Reid.
\newblock Rankings of partial derivatives.
\newblock In \emph{Proceedings of the 1997 international symposium on Symbolic
  and algebraic computation - {ISSAC}'97}, 1997.
\newblock URL \url{https://doi.org/10.1145/258726.258737}.

\bibitem[Simmons and Towsner(2019)]{ST19}
W.~Simmons and H.~Towsner.
\newblock Proof mining and effective bounds in differential polynomial rings.
\newblock \emph{Advances in Mathematics}, 343:\penalty0 567--623, 2019.
\newblock URL \url{https://doi.org/10.1016/j.aim.2018.11.026}.

\bibitem[van~den Dries and Shmidt(1984)]{vdDS84}
L.~van~den Dries and K.~Shmidt.
\newblock Bounds in the theory of polynomial rings over fields. a nonstandard
  approach.
\newblock \emph{Inventiones mathematicae}, 76:\penalty0 77--91, 1984.
\newblock URL \url{https://doi.org/10.1007/BF01388493}.

\end{thebibliography}
\end{document}